\documentclass[preprint,11pt]{elsarticle}
\usepackage{amssymb}
\usepackage{amsthm}
\usepackage{graphics}
\usepackage{epsfig}
\usepackage{amssymb}
\usepackage{graphicx}
\usepackage{subfigure}
\usepackage{color}
\usepackage{comment}
\usepackage{tikz}
\usepackage{tikz,pgfplots}
\pgfplotsset{compat=1.12,axis lines=center}
\usepgfplotslibrary{fillbetween}
\usetikzlibrary{patterns}
\usepackage{hyperref}
\usepackage{a4wide}
\usepackage{amsmath}
\usepackage{amsfonts}
\usepackage{enumerate}
\usepackage{multicol}

\makeatletter
\def\ps@pprintTitle{%
	\let\@oddhead\@empty
	\let\@evenhead\@empty
	\let\@oddfoot\@empty
	\let\@evenfoot\@oddfoot
}
\makeatother

\hypersetup{
	colorlinks=true,
	linkcolor=blue,
	filecolor=dviolet,      
	urlcolor=blue,
} 

\newtheorem*{theoremaux}{Theorem \theoremauxnum}
\gdef\theoremauxnum{1}

	\newtheorem{lemma}{\bf Lemma}[section]
	\newtheorem{example}{\bf Example}[section]
	\newtheorem{theorem}{\bf Theorem}[section]
	\newtheorem{proposition}[lemma]{\bf Proposition}
	\newtheorem{corollary}[lemma]{\bf Corollary}
	\newtheorem{definition}{\bf Definition}[section]
	\newtheorem{remark}{\bf Remark}[section]
	\newtheorem{conjecture}{\bf Conjecture}[section]
	
\journal{~}

\begin{document}

	\begin{frontmatter}
		
		
		
		\title{Existence of a Non-Zero $(0,1)$-Vector in the Row Space of Adjacency Matrices of Simple Graphs}

		
		
		\author{S. Bera}
		\ead{sudip\_bera@daiict.ac.in} 
		\address{Faculty of Mathematics,\\ 
			DA-IICT, Gandhinagr \\ India} 
		%

\begin{abstract}
We look for a non-zero $(0, 1)$-vector in the row space of the adjacency matrix $A(\Gamma)$ of a graph $\Gamma,$ provided $\Gamma$ has at least one edge. Akbari, Cameron, and Khosrovshahi conjectured that there exists a non-zero $(0,1)$-vector in the row space of $A(\Gamma)$ (over the real numbers) which does not occur as a row of $A(\Gamma).$ This conjecture can be easily verified for graphs having diameter is equal to $1$ (complete graphs). In this article, we affirmatively prove this conjecture for any graph whose diameter is $\geq 4.$  Furthermore, in the remaining two cases that is, for graphs with diameter is equal to $2$ or $3,$ we report some progress in support of the conjecture.
\end{abstract}  

\end{frontmatter}

\section{Introduction}
Graphs are among the most standard models of both natural and human-made structures. They can be used to model many types of relations and process dynamics in computer science, physical, biological, and social systems. Many problems of practical interest can be represented by graphs. In general, graph theory has a wide range of applications in diverse fields. To investigate the different properties of a graph, 
researchers have introduced several matrices associated with it. The most common matrices that have been studied for graphs are defined by associating the vertices with the rows and columns as follows: 
\begin{definition}[Adjacency matrix]
For a simple undirected graph $\Gamma$ with vertex set \[\{v_1,  \cdots, v_n\},\] the \emph{adjacency matrix} of the graph $\Gamma$ is $A(\Gamma)=(a_{ij})_{n\times n}$ is defined as follows: 
\[a_{ij}=
\begin{cases}
1,  &\text{ if } v_i \text{ adjacent to }  v_j\\
0, & \text{ otherwise }
\end{cases}.\]
\end{definition}
The \emph{rank} of an undirected graph $\Gamma,$ denoted by $r(\Gamma)$ is defined as the rank of its adjacency matrix. A graph is said to be \emph{singular} (\emph{non-singular}) if its adjacency matrix is a singular (non-singular) matrix. The eigenvalues $\lambda_1, \cdots, \lambda_n$ of $A(\Gamma)$ are said to be the eigenvalues of the graph $\Gamma,$ and to form the spectrum of this graph. The number of zero eigenvalues in the spectrum of the graph $\Gamma$ is called its \emph{nullity} and is denoted by $\eta(\Gamma).$ So, $\eta(\Gamma)=n-r(\Gamma).$ 

The rank of a graph is an important parameter, as several important graph-theoretic parameters are bounded by a function of the rank of a graph, and the rank is bounded by a function of the number $t$ of negative eigenvalues (by Theorem 12 of \cite{Cameron-akbaria-conjecture}). Moreover, characterizing all graphs $\Gamma$ in which $r(\Gamma)<n$ (equivalently $\eta(\Gamma)>0),$ is of great interest in chemistry, because, as has been shown in \cite{Longuet-Higgins}, for a bipartite graph $\Gamma$ (corresponding to an alternate hydrocarbon), if $r(\Gamma)<n,$ then it indicates the molecule which such a graph represents is unstable. In \cite{cha-of-singular-graph}, Collatz and Sinogowitz first posed this problem. The problem has not yet been solved completely, only for trees and bipartite graphs some particular results are known (see
\cite{Trees-max-nulity,nulity-line-graph-tree-G-S-1}). In recent years, this problem has been paid much attention by many researchers (\cite{Nulity-of-graphs-ELA,Nulity-tricyclic-graph-Cheng,Nulity-bipartite,Nulity-bicyclic-graph}).

On the other hand, the problem of bounding the order of a graph in terms of the rank was
first studied by Kotlov and Lov\'{a}sz \cite{LOvas}. Let $r\geq 2$ be an integer. Then the order of a graph with rank $r$ is trivially bounded above by $2r-1$ as
soon as we make the assumption that the graph is reduced, where a graph is said to be \emph{reduced} if no two vertices have the same set of neighbours. Let $m(r)$ be the maximum possible order of a reduced graph of rank $r.$ Kotlov and Lov\'{a}sz \cite{LOvas} proved that there exists constant $c$ such that the order of every reduced graph of rank $r$ is at most  $c\cdot2^{\frac{r}{2}},$ Moreover, for any $r\geq 2,$ they constructed a graph of rank $r$ and order $n(r),$ where
\begin{equation*}
n(r)=
\begin{cases}
2^{\frac{r+2}{2}}-2, & \text{ if } r \text{ is even } \\
5\cdot2^{\frac{r-3}{2}}-2, & \text{  if } r \text{ is  odd }
\end{cases}
\end{equation*}
Akbari et al. \cite{Cameron-akbaria-conjecture} conjectured that in fact $m(r)=n(r).$ In \cite{maximum-order-adjacency-matrices} Haemers and Peeters proved the conjecture for graphs containing an induced matching of size $\frac{r}{2}$ or an induced subgraph consisting of a matching of size $\frac{(r-3)}{2}$ and a cycle of length $3.$ In \cite{Ghorbani-combinatorica}, Ghorbani et al. showed that if the conjecture is not true, then there would be a counterexample of rank at most $46.$ It was also shown that the order of every reduced graph of rank $r$ is at most $8n(r)+14.$ Royle \cite{the-rank-cograph-Royal} proved that the rank of every reduced
graph containing no path of length $3$ as an induced subgraph is equal to the order. In \cite{Ghorbani-DiscM-max-order-tree}, the authors proved that every reduced tree of rank $r$ has at most $t(r)=\frac{3r}{2}-1$ vertices and characterized all reduced trees of rank $r$ and order $t(r).$ Besides, they proved that any reduced bipartite graph of rank $r$ has at most $b(r)=2^{\frac{r}{2}}+\frac{r}{2}-1$ vertices. Also, in \cite{Ghorbani-JGT-max-order-triangle-free-graph}, the authors proved that each and every reduced non-bipartite triangle-free graph of rank $r$ has at most $c(r)$
vertices, where $c(r)=3\cdot2^{\lfloor\frac{r}{2}\rfloor-2}+\lfloor\frac{r}{2}\rfloor.$
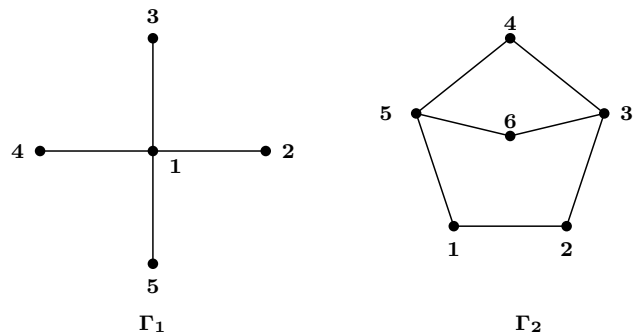
\begin{figure}
			\tiny
			\tikzstyle{ver}=[]
			\tikzstyle{vert}=[circle, draw, fill=black!100, inner sep=0pt, minimum width=4pt]
			\tikzstyle{vertex}=[circle, draw, fill=black!00, inner sep=0pt, minimum width=4pt]
			\tikzstyle{edge} = [draw,thick,-]
			\tikzstyle{node_style} = [circle,draw=blue,fill=blue!20!,font=\sffamily\Large\bfseries]
			\centering
			\begin{tikzpicture}[scale=1]
			\tikzstyle{edge_style} = [draw=black, line width=2mm, ]
			\tikzstyle{node_style} = [draw=blue,fill=blue!00!,font=\sffamily\Large\bfseries]
			\fill[black!100!](4,0) circle (.07);
			\fill[black!100!](5.5,0) circle (.07);
			\fill[black!100!](6,1.5) circle (.07);
			\fill[black!100!](4.75,2.5) circle (.07);
			\fill[black!100!](3.5,1.5) circle (.07);
			\draw[line width=.2 mm] (5.5,0) -- (4,0);
			\draw[line width=.2 mm] (4,0) -- (3.5,1.5);
			\draw[line width=.2 mm] (6,1.5) -- (5.5,0);
			\draw[line width=.2 mm] (4.75,2.5) -- (3.5,1.5);
			\draw[line width=.2 mm] (4.75,2.5) -- (6,1.5);
			%
			\draw[line width=.2 mm] (4.75,1.2) -- (6,1.5);
			\draw[line width=.2 mm] (4.75,1.2) -- (3.5,1.5);
			%
			\fill[black!100!](4.75,1.2) circle (.07);
			\node (B1) at (4,-.3)   {$\bf{1}$};
			\node (B2) at (5.5,-.3)  {$\bf{2}$};
			\node (B3) at (6.3,1.5)   {$\bf{3}$};
			\node (B4) at (4.75,2.7)   {$\bf{4}$};
			\node (A4) at (3.1,1.5)  {$\bf{5}$};
			\node (A4) at (4.75, 1.4)  {$\bf{6}$};
			\node (B1) at (.3,.8)   {$\bf{1}$};
			\node (B2) at (1.8,1)  {$\bf{2}$};
			\node (B3) at (0,2.8)   {$\bf{3}$};
			\node (B4) at (-1.8,1)   {$\bf{4}$};
			\node (A4) at (0,-.8)  {$\bf{5}$};
			\node (A4) at (0, -1.3)  {$\bf{\Gamma_1}$};
			\node (A4) at (5, -1.3)  {$\bf{\Gamma_2}$};
			\fill[black!100!](0,1) circle (.07);
			\fill[black!100!](1.5,1) circle (.07);
			\fill[black!100!](-1.5,1) circle (.07);
			\fill[black!100!](0,-.5) circle (.07);
			\fill[black!100!](0,2.5) circle (.07);
			\draw[line width=.2 mm] (0,1) -- (1.5,1);
			\draw[line width=.2 mm] (0,1) -- (-1.5,1);
			\draw[line width=.2 mm] (0,1) -- (0,2.5);
			\draw[line width=.2 mm] (0,1) -- (0,-.5);
			\end{tikzpicture}
			\caption{An illustration of the conjecture}
			\label{fig:example of conj}	
		\end{figure}
In \cite{Cameron-akbaria-conjecture} Akbari, Cameron, and Khosrovshahi proposed the following conjecture about the row space of graphs. 
\begin{conjecture}[Question 2, \cite{Cameron-akbaria-conjecture}]\label{conj: Cameron on row space}
Let $\Gamma$ be a graph with at least one edge. If $A(\Gamma)$ is the adjacency matrix of $\Gamma,$ then there exists a non-zero $(0,1)$-vector in the row space of $A(\Gamma)$ (over the real numbers) which does not occur as a row of $A(\Gamma).$
\end{conjecture}
If Conjecture \ref{conj: Cameron on row space} is true for all graphs with at least one edge, then by using Theorem $3$ of \cite{rank-vertex-addition} one can see that the function $m(r)$ is an increasing function. This reflects a growing complexity or ``capacity" for graphs with higher rank to accommodate more vertices while still adhering to the conditions imposed by the rank. The increasing nature of $m(r)$ supports the conjecture $m(r)=n(r)$ by demonstrating that the size of graphs with a given rank grows predictably, ensuring that the rank constraints and graph sizes are tightly coupled. Moreover,
This problem bridges graph theory and linear algebra, specifically focusing on how linear combinations of rows can lead to new insights about the graph's structure. Investigating the existence of non-zero $(0,1)$-vectors in the row space could lead to a deeper understanding of how adjacency matrices can encode graph-theoretic information and how graphs with certain properties (such as connectivity or degree distribution) might influence the rank and row space of their adjacency matrices.
\begin{example}
Consider the graphs in Figure \ref{fig:example of conj}. The adjacency matrices of these graphs are \[A(\Gamma_1)=\label{mat: example for conj}
\left(
\begin{array}{ccccc}
		0 &1 & 1&1&1\\
		1 &0 & 0&0&0\\
		1 &0 & 0&0&0\\
		1 &0 & 0&0&0\\
		1 &0 & 0&0&0
		\end{array}
		\right),
		A(\Gamma_2)=	\left(
		\begin{array}{cccccc}
		0 &1 & 0&0&1&0\\
		1 &0 & 1&0&0&0\\
		0 &1 & 0&1&0&1\\
		0 &0 & 1&0&1&0\\
		1 &0 & 0&1&0&1\\
		0 &0& 1&0&1&0 
		\end{array}
		\right)
		.\]
Notice that in the matrix $A(\Gamma_1), R_1+R_2=(1, 1, 1, 1, 1)$ but $(1, 1, 1, 1, 1)$ does not occur as a row of $A(\Gamma_1).$ In case of $A(\Gamma_2), R_1+R_2=(1,1,1,0,1,0)$ but $(1,1,1,0,1,0)$ does not occur as a row of $A(\Gamma_2).$ So, both the matrices $A(\Gamma_1)$ and $A(\Gamma_2)$ satisfy the conjecture. 
\end{example}
In this article, we prove that Conjecture \ref{conj: Cameron on row space} is true for all graphs with diameter $\geq 4.$  Besides, in the remaining two cases that is, for graphs with diameter is equal to $2$ or $3,$ we report some progress in support of the conjecture.
	
The plan of the paper is as follows. In Section \ref{Sec: basic, notation and preliminaries}, we mention some earlier known results. The main results of this article are described in Section \ref{Sec: main results}. Finally, the proofs of the main results are given in Section \ref{Sec: proof of main thm}.
\section{Basic definitions, Notations and Preliminaries}\label{Sec: basic, notation and preliminaries}
For the convenience of the reader and also for later use, we recall some basic definitions and notations about graphs. Let $\Gamma$ be a graph with vertex set $V(\Gamma).$ The \emph{order} and \emph{size} of a graph $\Gamma$ are $|V(\Gamma)|$ and $|E(\Gamma)|$ respectively, where $E(\Gamma)$ is the edge set of $\Gamma.$ Two elements $u$ and $v$ are said to be adjacent (denoted by $u\sim v)$ if there is an edge between them. For a vertex $u,$ we denote by $\text{nbd}(u),$ the set of vertices that are adjacent to $u.$ The degree of $u,$ denoted by $\text{deg}(u),$ is the cardinality of the set $\text{nbd}(u).$ A vertex $v$ is said to be \emph{pendant vertex} if $\text{deg}(v)=1.$ An edge of a graph is said to be \emph{pendant} if one of its vertices is a pendant vertex. A \emph{path} of length $k,$ denoted by $P_{k+1}$ between two vertices $v_0$ and $v_k$ is an alternating sequence of vertices and edges $v_0, e_0, v_1, e_1, v_2, \cdots , v_{k-1}, e_{k-1}, v_k$, where the $v_i'$s are distinct
(except possibly the first and last vertices) and $e_i'$s are the edges $(v_i, v_{i+1}).$  A graph $\Gamma$ is said to be \emph{connected} if for any pair of vertices $u$ and $v,$ there exists a path between $u$ and $v.$ The distance between two vertices $u$ and $v$ in a connected graph $\Gamma$ is the length of the shortest path between them and it is denoted by $d(u, v).$ Clearly, if $u$ and $v$ are adjacent, then $d(u, v)=1.$ For a graph $\Gamma,$ its \emph{diameter} is defined as $\text{diam}(\Gamma)= \max_{u, v \in V} d(u, v).$ That is, the diameter of graph is the largest possible distance between pair of vertices of a graph. $\Gamma$ is said to be \emph{complete} if any two distinct vertices are adjacent. A vertex of a graph $\Gamma$ is called a \emph{dominating vertex} if it is adjacent to every other vertex. The \emph{star} graph is the complete bipartite graph $K_{1, n}, n\in \mathbb{N}.$ For a matrix $M,$ by $R(M)$ and $r(M),$ we mean the row space (over the field $\mathbb{R})$ and the rank of the matrix $M$ respectively. The $i^{\text{th}}$ row of a matrix is denoted by $R_i.$ $M^t$ denotes the transpose of the matrix $M.$ We denote the set $\{1, \cdots, n\}$ by $[n].$ 
	
Now, we shift our attention to recall some earlier known  results on graph operations. For that first we need to define a graph operation called \emph{multiplication of vertices} (see p. 53 of \cite{Graph-multiplication-Golumbic}). 
\begin{definition}[\cite{Graph-multiplication-Golumbic}]\label{def: multiplication of vetex}
Let $m=(m_1, \cdots, m_n)$ be a vector of positive integers. Let $\Gamma$ be a graph with the vertex set $V(\Gamma)=\{v_1, \cdots, v_n\}.$ Then the graph, denoted by $\Gamma\bigodot m,$ is obtained from $\Gamma$ by replacing each vertex
$v_i$ of $\Gamma$ with an independent set of $m_i$ vertices $v^1_{i}\cdots, v^{m_i}_{i}$ and joining $v^s_i$ with $v^t_j$ if and only if $v_i$ and $v_j$ are adjacent in $\Gamma.$ We say that $V_i=\{v^1_{i}, \cdots, v^{m_i}_{i} \}$ is the set of vertices of $\Gamma\bigodot m$ corresponding to $v_i.$ The resulting graph $\Gamma
\bigodot m$ is said to be obtained from $\Gamma$ by multiplication of vertices. By multiplicating the single vertex $v_i$ in $\Gamma$ we shall mean a new graph $\Gamma'$ such that $\Gamma'=\Gamma\bigodot m,$ where $m=(\underbrace{1,\cdots, 1,}_{i-1 \text{ entries }}m_i, 1, \cdots, 1), m_i\geq 1.$ By \emph{duplicating} (defined in \cite{Henning-Michael}) the vertex $v_i$ in $\Gamma$ we shall mean the multiplication $\Gamma\bigodot m,$ where $m=(\underbrace{1,\cdots, 1,}_{i-1 \text{ entries }}2, 1, \cdots, 1),$ (that is, adding a new vertex to the graph $\Gamma$ and joining it to every vertex in $\text{nbd}(v_i)).$ 
If $\Gamma$ contains a vertex of degree $2,$ then we define \emph{degree-$2$ vertex duplication} on $\Gamma$ as the operation that produces a new graph from $\Gamma$ by duplicating any vertex of degree $2.$ Let $\mathcal{M}(\Gamma_1, \cdots, \Gamma_k)$ be the collection of all graphs $\Gamma'$ which can be constructed from one of the graphs in $\{\Gamma_1, \cdots, \Gamma_k\}$	by multiplication of vertices.	
	\end{definition}
	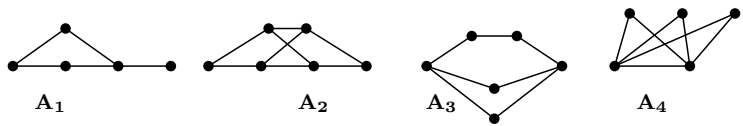
\begin{figure}
		\tiny
		\tikzstyle{ver}=[]
		\tikzstyle{vert}=[circle, draw, fill=black!100, inner sep=0pt, minimum width=4pt]
		\tikzstyle{vertex}=[circle, draw, fill=black!00, inner sep=0pt, minimum width=4pt]
		\tikzstyle{edge} = [draw,thick,-]
		\tikzstyle{node_style} = [circle,draw=blue,fill=blue!20!,font=\sffamily\Large\bfseries]
		\centering
		\begin{tikzpicture}[scale=1]
		\tikzstyle{edge_style} = [draw=black, line width=2mm, ]
		\tikzstyle{node_style} = [draw=blue,fill=blue!00!,font=\sffamily\Large\bfseries]
		\fill[black!100!](0,0) circle (.07);
		\fill[black!100!](.7,0) circle (.07);
		\fill[black!100!](1.4,0) circle (.07);
		\fill[black!100!](2.1,0) circle (.07);
		\fill[black!100!](.7,.5) circle (.07);
		\draw[line width=.2 mm] (0,0) -- (.7,0);
		\draw[line width=.2 mm] (.7,0) -- (1.4,0);
		\draw[line width=.2 mm] (2.1,0) -- (1.4,0);
		\draw[line width=.2 mm] (.7,.5) -- (0,0);
		\draw[line width=.2 mm] (.7,.5) -- (1.4,0);
		\fill[black!100!](2.6,0) circle (.07);
		\fill[black!100!](3.3,0) circle (.07);
		\fill[black!100!](4,0) circle (.07);
		\fill[black!100!](4.7,0) circle (.07);
		\fill[black!100!](3.4,.5) circle (.07);
		\fill[black!100!](3.9,.5) circle (.07);
		\draw[line width=.2 mm] (2.6,0) -- (3.3,0);
		\draw[line width=.2 mm] (3.3,0) -- (4,0);
		\draw[line width=.2 mm] (4,0) -- (4.7,0);
		\draw[line width=.2 mm] (3.4,.5) -- (2.6,0);
		\draw[line width=.2 mm] (3.4,.5) -- (4,0);
		\draw[line width=.2 mm] (3.4,.5) -- (3.9,.5);
		\draw[line width=.2 mm] (3.3,0) -- (3.9,.5);
		\draw[line width=.2 mm] (4.7,0) -- (3.9,.5);
		\fill[black!100!](5.5,0) circle (.07);
		\fill[black!100!](6.1,.4) circle (.07);
		\fill[black!100!](6.7,.4) circle (.07);
		\fill[black!100!](7.3,0) circle (.07);
		\fill[black!100!](6.4,-.3) circle (.07);
		\fill[black!100!](6.4,-.7) circle (.07);
		\draw[line width=.2 mm] (5.5,0) -- (6.1,.4);
		\draw[line width=.2 mm] (6.1,.4) -- (6.7,.4);
		\draw[line width=.2 mm] (6.7,.4) -- (7.3,0);
		\draw[line width=.2 mm] (7.3,0) -- (6.4,-.3);
		\draw[line width=.2 mm] (6.4,-.3) -- (5.5,0);
		\draw[line width=.2 mm] (6.4,-.7) -- (5.5,0);
		\draw[line width=.2 mm] (6.4,-.7) -- (7.3,0);
		\fill[black!100!](8,0) circle (.07);
		\fill[black!100!](9,.0) circle (.07);
		\fill[black!100!](8.2,.7) circle (.07);
		\fill[black!100!](8.9,.7) circle (.07);
		\fill[black!100!](9.6,.7) circle (.07);
		\draw[line width=.2 mm] (8,0) -- (9,.0);
		\draw[line width=.2 mm] (8.2,.7) -- (8,0);
		\draw[line width=.2 mm] (8.2,.7) -- (9,0);
		\draw[line width=.2 mm] (8.9,.7) -- (8,0);
		\draw[line width=.2 mm] (8.9,.7) -- (9,0);
		\draw[line width=.2 mm] (9.6,.7) -- (8,0);
		\draw[line width=.2 mm] (9.6,.7) -- (9,0);
		\node (A4) at (.5,-.5)  {$\bf{A_1}$};
		\node (A4) at (4,-.5)  {$\bf{A_2}$};
		\node (A4) at (5.7,-.5)  {$\bf{A_3}$};
		\node (A4) at (8.5,-.5)  {$\bf{A_4}$};
		\end{tikzpicture}
		\caption{Explanation of multiplication of vertices}
		\label{fig:exampl for vertex multiplication in graphs}	
\end{figure}
\begin{example}
Note that, the graphs $A_1, A_2\in \mathcal{M}(P_4), A_3\in \mathcal{M}(C_5)$ and $A_4\in \mathcal{M}(K_3).$	
\end{example}
Some of the important properties of the graph $\Gamma\bigodot m$ are described below. We will use these properties to prove Conjecture \ref{conj: Cameron on row space}. 
\begin{lemma}[\cite{long-graph-given-rank}] \label{lem: diam fixed by multiplication vertices}
Let $\Gamma\bigodot m$ be a multiplication of a graph $\Gamma.$ If $\Gamma$ is not a complete graph, then $\text{diam}(\Gamma\bigodot m)=\text{diam}(\Gamma).$	
\end{lemma}
\begin{remark}\label{rem: diameter of graph after multi of comple graph}
For the complete graph $K_n,$ $\text{diam}(K_n)=1$ and $\text{diam}(K_n\bigodot m)=2.$ 
\end{remark}
One crucial observation on the rank of the graphs $\Gamma$ and $\Gamma',$ (where $\Gamma'\in \mathcal{M}(\Gamma))$ is as follows.
\begin{proposition}[\cite{LAA-rank-5-graph-charecterization}]\label{Prop: rank of graph and its multipl aresame}	
For graphs $\Gamma$ and $\Gamma'$ if $\Gamma'\in\mathcal{M}(\Gamma),$ then $r(\Gamma')=r(\Gamma).$
\end{proposition}
In this portion, we construct a special family of diameter-$2$ graph by duplicating the vertices. In Section \ref{Sec: proof of main thm}, we will show that this family satisfies the conjecture. 
\begin{example}\label{Exam: duplication and family H}
In this example, we construct a family of diameter $2$ graphs from the cycle $C_5$ and the graph $\Gamma_0$ (where $\Gamma_0$ (depicted in Figure \ref{fig:graph of diam 2;cycle 0f length $5$ and its duplication}) be the graph obtained from a $3$-cycle by adding a pendant edge to each vertex of the cycle and then adding a new vertex and joining it to the three degree $1$
vertices) by duplicating any vertex of degree $2.$ Let $\mathcal{H}$ be the family of graphs that: 
\begin{enumerate}
\item contains $C_5, \Gamma_0$ and the Petersen graph; and 
\item is closed under degree-$2$ vertex duplication,
\end{enumerate}
\begin{figure}
\tiny
\tikzstyle{ver}=[]
			\tikzstyle{vert}=[circle, draw, fill=black!100, inner sep=0pt, minimum width=4pt]
			\tikzstyle{vertex}=[circle, draw, fill=black!00, inner sep=0pt, minimum width=4pt]
			\tikzstyle{edge} = [draw,thick,-]
			\tikzstyle{node_style} = [circle,draw=blue,fill=blue!20!,font=\sffamily\Large\bfseries]
			\centering
			\begin{tikzpicture}[scale=1]
			\tikzstyle{edge_style} = [draw=black, line width=2mm, ]
			\tikzstyle{node_style} = [draw=blue,fill=blue!00!,font=\sffamily\Large\bfseries]
			\node (A4) at (.7,-.5)  {$\bf{\Gamma_1}$};
			\node (A4) at (4.7,-.5)  {$\bf{\Gamma_2}$};
			\node (A4) at (9.7,-.5)  {$\bf{\Gamma_3}$};
			\node (A4) at (1.5,-3.5)  {$\bf{\Gamma_0}$};
			\node (A4) at (6.5,-3.5)  {$\bf{\Gamma_4}$};
			\node (A4) at (11,-3.5)  {$\bf{\Gamma_5}$};
			\fill[red!100!](0,0) circle (.07);
			\fill[red!100!](1.5,0) circle (.07);
			\fill[black!100!](2,1.5) circle (.07);
			\fill[black!100!](0.75,2.5) circle (.07);
			\fill[black!100!](-.5,1.5) circle (.07);
			%
			%
			\draw[line width=.2 mm] (1.5,0) -- (0,0);
			\draw[line width=.2 mm] (0,0) -- (-0.5,1.5);
			\draw[line width=.2 mm] (2,1.5) -- (1.5,0);
			\draw[line width=.2 mm] (.75,2.5) -- (-0.5,1.5);
			\draw[line width=.2 mm] (.75,2.5) -- (2,1.5);
			%
			\fill[red!100!](4,0) circle (.07);
			\fill[red!100!](5.5,0) circle (.07);
			\fill[black!100!](6,1.5) circle (.07);
			\fill[black!100!](4.75,2.5) circle (.07);
			\fill[black!100!](3.5,1.5) circle (.07);
			\draw[line width=.2 mm] (5.5,0) -- (4,0);
			\draw[line width=.2 mm] (4,0) -- (3.5,1.5);
			\draw[line width=.2 mm] (6,1.5) -- (5.5,0);
			\draw[line width=.2 mm] (4.75,2.5) -- (3.5,1.5);
			\draw[line width=.2 mm] (4.75,2.5) -- (6,1.5);
			\draw[line width=.2 mm] (4.75,.75) -- (6,1.5);
			\draw[line width=.2 mm] (4.75,.75) -- (3.5,1.5);
			\draw[line width=.2 mm] (4.75,1.2) -- (6,1.5);
			\draw[line width=.2 mm] (4.75,1.2) -- (3.5,1.5);
			\draw[line width=.2 mm] (4.75,1.8) -- (6,1.5);
			\draw[line width=.2 mm] (4.75,1.8) -- (3.5,1.5);
			\fill[black!100!](4.75,.75) circle (.07);
			\fill[black!100!](4.75,1.2) circle (.07);
			\fill[black!100!](4.75,1.8) circle (.07);
			%
			\fill[red!100!](9,0) circle (.07);
			\fill[red!100!](10.5,0) circle (.07);
			\fill[black!100!](11,1.5) circle (.07);
			\fill[black!100!](9.75,2.5) circle (.07);
			\fill[black!100!](8.5,1.5) circle (.07);
			\draw[line width=.2 mm] (10.5,0) -- (9,0);
			\draw[line width=.2 mm] (9,0) -- (8.5,1.5);
			\draw[line width=.2 mm] (11,1.5) -- (10.5,0);
			\draw[line width=.2 mm] (9.75,2.5) -- (8.5,1.5);
			\draw[line width=.2 mm] (9.75,2.5) -- (11,1.5);
			\draw[line width=.2 mm] (8,1.5) -- (9,0);
			\draw[line width=.2 mm] (8,1.5) -- (9.75,2.5);
			\fill[black!100!](8,1.5) circle (.07);
			\fill[black!100!](7.5,1.5) circle (.07);
			\draw[line width=.2 mm] (7.5,1.5) -- (9,0);
			\draw[line width=.2 mm] (7.5,1.5) -- (9.75,2.5);
			\draw[line width=.2 mm] (11.5,1.5) -- (10.5,0);
			\draw[line width=.2 mm] (11.5,1.5) -- (9.75,2.5);
			\fill[black!100!](11.5,1.5) circle (.07);
			\draw[line width=.2 mm] (12,1.5) -- (10.5,0);
			\draw[line width=.2 mm] (12,1.5) -- (9.75,2.5);
			\fill[black!100!](12,1.5) circle (.07);	
			\draw[line width=.2 mm] (8.8,1.5) -- (9,0);
			\draw[line width=.2 mm] (8.8,1.5) -- (9.75,2.5);
			\fill[black!100!](8.8,1.5) circle (.07);
			\draw[line width=.2 mm] (10.7,1.5) -- (10.5,0);
			\draw[line width=.2 mm] (10.7,1.5) -- (9.75,2.5);
			\fill[black!100!](10.7,1.5) circle (.07);
			\draw[line width=.2 mm] (0,-3) -- (3,-3);
			\draw[line width=.2 mm] (0,-3) -- (1.5,-1);
			\draw[line width=.2 mm] (3,-3) -- (1.5,-1);
			\fill[black!100!](0,-3) circle (.07);
			\fill[black!100!](3,-3) circle (.07);
			\fill[black!100!](1.5,-1) circle (.07);
			\draw[line width=.2 mm] (0,-3) -- (1.5,-2.5);
			\draw[line width=.2 mm] (1.5,-1) -- (1.5,-2.5);
			\draw[line width=.2 mm] (3,-3) -- (1.5,-2.5);
			\fill[red!100!](1.5,-2.5) circle (.07);
			\fill[red!100!](.8,-2.74) circle (.07);
			\fill[black!100!](2.2,-2.74) circle (.07);
			\fill[black!100!](1.5,-1.9) circle (.07);
			%
			\draw[line width=.2 mm] (5,-3) -- (8,-3);
			\draw[line width=.2 mm] (5,-3) -- (6.5,-1);
			\draw[line width=.2 mm] (8,-3) -- (6.5,-1);
			\fill[black!100!](5,-3) circle (.07);
			\fill[black!100!](8,-3) circle (.07);
			\fill[black!100!](6.5,-1) circle (.07);
			\draw[line width=.2 mm] (5,-3) -- (6.5,-2.5);
			\draw[line width=.2 mm] (6.5,-1) -- (6.5,-2.5);
			\draw[line width=.2 mm] (8,-3) -- (6.5,-2.5);
			\fill[red!100!](6.5,-2.5) circle (.1);
			\fill[red!100!](5.8,-2.74) circle (.07);
			\fill[black!100!](7.2,-2.74) circle (.07);
			\fill[black!100!](6.5,-1.9) circle (.07);
			\fill[black!100!](6.25,-1.9) circle (.07);
			\fill[black!100!](6.75,-1.9) circle (.07);
			\draw[line width=.2 mm] (6.25,-1.9) -- (6.5,-1);
			\draw[line width=.2 mm] (6.25,-1.9) -- (6.5,-2.5);
			\draw[line width=.2 mm] (6.75,-1.9) -- (6.5,-1);
			\draw[line width=.2 mm] (6.75,-1.9) -- (6.5,-2.5);
			\fill[black!100!](5.8,-2.4) circle (.07);
			\draw[line width=.2 mm] (6.5,-2.5) -- (5.8,-2.4);
			\draw[line width=.2 mm] (5,-3) -- (5.8,-2.4);
			\fill[black!100!](6,-2.85) circle (.06);
			\draw[line width=.2 mm] (6.5,-2.5) -- (6,-2.85);
			\draw[line width=.2 mm] (5,-3) -- (6,-2.85);
			\fill[black!100!](7.2,-2.4) circle (.07);
			\draw[line width=.2 mm] (6.5,-2.5) -- (7.2,-2.4);
			\draw[line width=.2 mm] (8,-3) -- (7.2,-2.4);
			\fill[black!100!](7,-2.85) circle (.06);
			\draw[line width=.2 mm] (6.5,-2.5) -- (7,-2.85);
			\draw[line width=.2 mm] (8,-3) -- (7,-2.85);
			\fill[red!100!](10,-3) circle (.07);
			\fill[red!100!](11.5,-3) circle (.07);
			\fill[black!100!](12,-1.5) circle (.07);
			\fill[black!100!](10.75,-.5) circle (.07);
			\fill[black!100!](9.5,-1.5) circle (.07);
			%
			%
			\draw[line width=.2 mm] (11.5,-3) -- (10,-3);
			\draw[line width=.2 mm] (10,-3) -- (9.5,-1.5);
			\draw[line width=.2 mm] (12,-1.5) -- (11.5,-3);
			\draw[line width=.2 mm] (10.75,-.5) -- (9.5,-1.5);
			\draw[line width=.2 mm] (10.75,-.5) -- (12,-1.5);	
			\fill[black!100!](10.3,-2.5) circle (.07);
			\fill[black!100!](11.2,-2.5) circle (.07);
			\fill[black!100!](10.3,-1.7) circle (.07);
			\fill[black!100!](10.75,-1.2) circle (.07);
			\fill[black!100!](11.2,-1.7) circle (.07);
			\draw[line width=.2 mm] (10.3,-2.5) -- (11.2,-1.7);
			\draw[line width=.2 mm]  (11.2,-1.7) -- (10.3,-1.7);
			\draw[line width=.2 mm] (10.3,-2.5) -- (10.75,-1.2);
			\draw[line width=.2 mm] (11.2,-2.5) -- (10.75,-1.2);	
			\draw[line width=.2 mm] (11.2,-2.5) -- (10.3,-1.7);	
			\draw[line width=.2 mm] (10.3,-2.5) -- (10,-3);	
			\draw[line width=.2 mm]  (11.2,-1.7) -- (12,-1.5);
			\draw[line width=.2 mm] (10.75,-.5) -- (10.75,-1.2);
			\draw[line width=.2 mm] (11.2,-2.5) -- (11.5,-3);	
			\draw[line width=.2 mm] (9.5,-1.5) -- (10.3,-1.7);	
\end{tikzpicture}
\caption{ $\Gamma_1=C_5, \Gamma_2$ is the graph obtained from $C_5$ by duplicating three degree-$2$ vertices, and $\Gamma_3$ is obtained from $C_5$ by duplicating six degree-$2$ vertices. Similarly, the graph $\Gamma_4$ is obtained from $\Gamma_0$ by duplicating six degree-$2$ vertices and $\Gamma_5$ is the Petersen graph.}
\label{fig:graph of diam 2;cycle 0f length $5$ and its duplication}	
\end{figure}
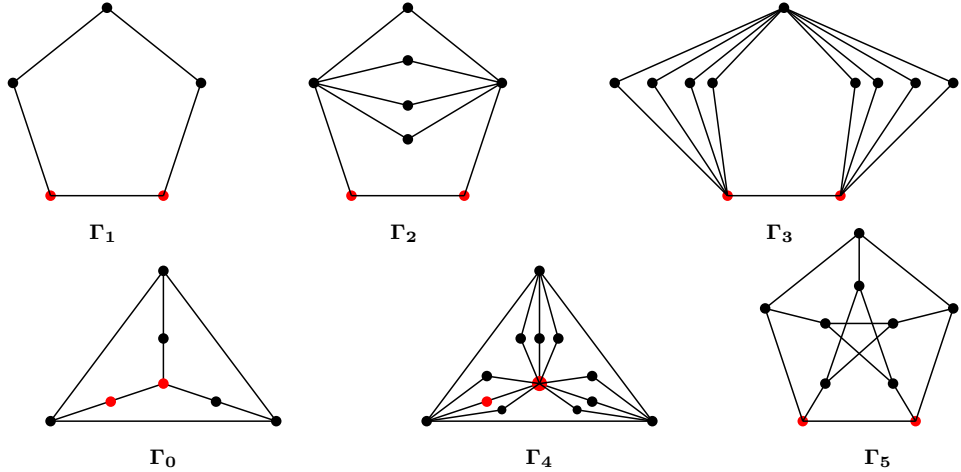	
\end{example}
In \cite{ERDOS-RENYEL}, Erd\H{o}s and R\'{e}nyi, proved the classical result (see Theorem \ref{Thm: s geq 2n-5, by ERDOS-RENYEL}) on the minimum size of a diameter-$2$ graph with no dominating vertex.
\begin{theorem}[\cite{ERDOS-RENYEL}]\label{Thm: s geq 2n-5, by ERDOS-RENYEL}
If $\Gamma$ is a diameter-$2$ graph of order $n$ and size $s$ with no dominating vertex, then $s\geq 2n-5.$
\end{theorem}
After that in \cite{Henning-Michael}, Henning and Southey classified all graphs $\Gamma$ such that $s=2n-5.$ In fact, they proved the following:
\begin{theorem}[\cite{Henning-Michael}]\label{thm H-M: dim2 m  geq 2n-5}
If $\Gamma$ is a diameter-$2$ graph of order $n$ and size $s$ with no dominating vertex, then $s\geq 2n-5$ with equality if and only if $\Gamma\in\mathcal{H}.$ 
\end{theorem}
We next prove some important results which are used to prove our main theorems. 
\begin{lemma}\label{lem: nbd umeet nbd v empty after multiplication}
Let $\Gamma$ be a simple connected graph with $|V(\Gamma)|=n.$ Suppose $\Gamma$ has two vertices $u, v$ such that $u\sim v$ and $\text{nbd}(u)\cap \text{nbd}(v)=\emptyset.$ Then also in the graph $\Gamma\bigodot m,$ there exist two vertices $u', v'$ such that $u'\sim v'$ and $\text{nbd}(u')\cap \text{nbd}(v')=\emptyset,$ where $m=(m_1, \cdots, m_n),$ and $m_i\geq 1,$ for all $i\in [n].$
\end{lemma}
\begin{proof}
Definition \ref{def: multiplication of vetex} implies that there exist two vertices $u', v'$ such that $u'\sim v'$ in $\Gamma\bigodot m.$ Let $v''\in V(\Gamma\bigodot m)$ such that $v''\in \text{nbd}(u')\cap\text{nbd}(v')$ in $\Gamma\bigodot m.$ Then, $u'\sim v''\sim v'$ in $\Gamma\bigodot m.$ Therefore, by Definition \ref{def: multiplication of vetex}, $\Gamma$ must have a vertex $w$ (say) such that $u\sim w\sim v.$ That is, $w\in \text{nbd}(u)\cap\text{nbd}(v)$ in $\Gamma,$ a contradiction. 	
\end{proof}
Now we propose a lemma regarding the row space of the adjacency matrix $A(\Gamma\bigodot m)$ of the graph $\Gamma\bigodot m,$ where $m=(m_1, \cdots, m_n), m_i\geq 1 \text{ for each } i\in [n].$
\begin{lemma}\label{lem: v is in row space of graph impl v,some extra is in row space of graph with multiplication}
Let $\Gamma$ be a simple connected graph such that $V(\Gamma)=\{v_1, \cdots v_n\}.$ Let $V=(x_1, \cdots, x_n)\in R(A(\Gamma)),$ where $x_i=0$ or $1,$ for each $i\in [n].$ Then for any $m=(m_1, \cdots, m_n),$ the vector $V'=(\underbrace{x_1, \cdots, x_1}_{m_1 \text{ times }}, \underbrace{x_2, \cdots, x_2}_{m_2 \text{ times }}, \cdots, \underbrace{x_n, \cdots, x_n}_{m_n \text{ times }})\in R(A(\Gamma\bigodot m)).$ 	
\end{lemma}
\begin{proof}
Let	
\begin{equation}
A(\Gamma)=	\left(
\begin{array}{cccc}
0 &a_{1,2} &\cdots&a_{1,n}\\
a_{2,1}&0 &\cdots&a_{2,n}\\
\vdots &\vdots&\ddots&\vdots\\
a_{n,1}&a_{n,2} &\cdots&0 
\end{array}
\right) 
\end{equation}
with respect to the vertex order $v_1, v_2, \cdots, v_n,$
where $a_{i, j}\in\{0, 1\},$ for all $i, j\in [n].$	
Since, $(x_1, \cdots, x_n)\in R(A(\Gamma)),$ we have $c_1, \cdots, c_n\in \mathbb{R}$ such that 
\begin{equation}
\left(
\begin{array}{cccc}
0 &a_{1,2} &\cdots&a_{1,n}\\
a_{2,1}&0 &\cdots&a_{2,n}\\
\vdots &\vdots&\ddots&\vdots\\
a_{n,1}&a_{n,2} &\cdots&0 
\end{array}
\right)	\left(\begin{array}{cccc}
c_1\\
c_2\\
\vdots\\
c_n
\end{array}
\right)=\left(\begin{array}{cccc}
x_1\\
x_2\\
\vdots\\
x_n
\end{array}
\right).
\end{equation}
Now, we arrange the vertices of the graph $\Gamma\bigodot m$ in the following order:
\[v_1^{1}, \cdots, v_1^{m_1}, v_2^{1}, \cdots, v_2^{m_2}, \cdots, v_n^{1}, \cdots, v_n^{m_n}.\]
Then using Definition \ref{def: multiplication of vetex}, it can be shown that \begin{equation}
A(\Gamma\bigodot m)=\left(
\begin{array}{cccc}
A_{11} &A_{12} &\cdots&A_{1n}\\
A_{21}&A_{22} &\cdots&A_{2n}\\
\vdots &\vdots&\ddots&\vdots\\
A_{n1}&A_{n2} &\cdots&A_{nn} 
\end{array}
\right), \text{ where }
\end{equation}
\begin{equation}
A_{ij}=
\begin{cases}
O_{m_i\times m_j}, & \text{ if }  a_{i,j}=0 \\
E_{m_i\times m_j}, & \text{ if }  a_{i,j}=1 
\end{cases}
\end{equation}
and by $O_{m_i\times m_j}$ (and $E_{m_i\times m_j})$ we mean that the matrix in which the $(i,j)^{\text{th}}$ entry is $0$ (and $1)$ for all $i, j\in [n].$ Let 
$C_i=\left(\begin{array}{cccc}
c_i\\
0\\
\vdots\\
0
\end{array}
\right) \text{ and } X_i=\left(\begin{array}{cccc}
x_i\\
x_i\\
\vdots\\
x_i
\end{array}
\right)$
be two $m_i\times 1$ matrices, for each $i\in [n].$ Then it is easy to see that,
\begin{equation}
\left(
\begin{array}{cccc}
A_{11} &A_{12} &\cdots&A_{1n}\\
A_{21}&A_{22} &\cdots&A_{2n}\\
\vdots &\vdots&\ddots&\vdots\\
A_{n1}&A_{n2} &\cdots&A_{nn} 
\end{array}
\right)\left(\begin{array}{cccc}
C_1\\
C_2\\
\vdots\\
C_n
\end{array}
\right)=\left(\begin{array}{cccc}
X_1\\
X_2\\
\vdots\\
X_n
\end{array}
\right).
\end{equation}
Thus, the vector $V'=(\underbrace{x_1, \cdots, x_1}_{m_1 \text{ times }}, \underbrace{x_2, \cdots, x_2}_{m_2 \text{ times }}, \cdots, \underbrace{x_n, \cdots, x_n}_{m_n \text{ times }})\in R(A(\Gamma\bigodot m)).$ 
\end{proof}
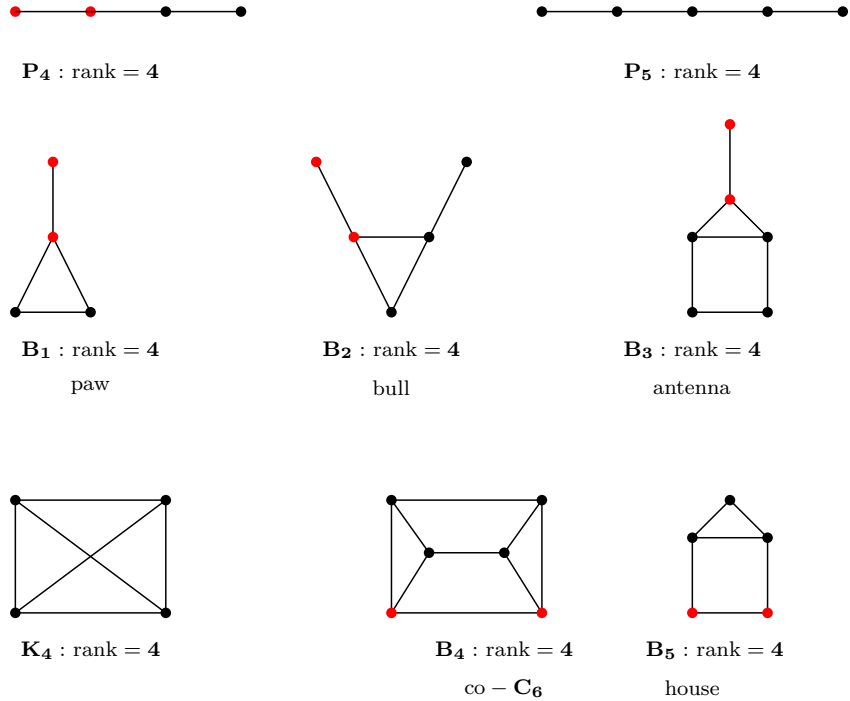
\begin{figure}
		\tiny
		\tikzstyle{ver}=[]
		\tikzstyle{vert}=[circle, draw, fill=black!100, inner sep=0pt, minimum width=4pt]
		\tikzstyle{vertex}=[circle, draw, fill=black!00, inner sep=0pt, minimum width=4pt]
		\tikzstyle{edge} = [draw,thick,-]
		\tikzstyle{node_style} = [circle,draw=blue,fill=blue!20!,font=\sffamily\Large\bfseries]
		\centering
		\begin{tikzpicture}[scale=1]
		\tikzstyle{edge_style} = [draw=black, line width=2mm, ]
		\tikzstyle{node_style} = [draw=blue,fill=blue!00!,font=\sffamily\Large\bfseries]
		\fill[red!100!](0,0) circle (.07);
		\fill[red!100!](1,0) circle (.07);
		\fill[black!100!](2,0) circle (.07);
		\fill[black!100!](3,0) circle (.07);
		\fill[black!100!](7,0) circle (.07);
		\fill[black!100!](8,0) circle (.07);
		\fill[black!100!](9,0) circle (.07);
		\fill[black!100!](10,0) circle (.07);
		\fill[black!100!](11,0) circle (.07);
		%
		\node (A4) at (1,-.8)  {$\bf{P_4:} \text{ rank}=4$};
		\node (A4) at (9,-.8)  {$\bf{P_5:} \text{ rank}=4$};
		\draw[line width=.2 mm] (0,0) -- (1,0);
		\draw[line width=.2 mm] (1,0) -- (2,0);
		\draw[line width=.2 mm] (2,0) -- (3,0);
		%
		\draw[line width=.2 mm] (7,0) -- (8,0);
		\draw[line width=.2 mm] (8,0) -- (9,0);
		\draw[line width=.2 mm] (9,0) -- (10,0);
		\draw[line width=.2 mm] (10,0) -- (11,0);
		\draw[line width=.2 mm] (0,-4) -- (1,-4);
		\draw[line width=.2 mm] (1,-4) -- (.5,-3);
		\draw[line width=.2 mm] (0,-4) -- (.5,-3);
		\draw[line width=.2 mm] (.5,-2) -- (.5,-3);
		\fill[black!100!](0,-4) circle (.07);
		\fill[black!100!](1,-4) circle (.07);
		\fill[red!100!](.5,-3) circle (.07);
		\fill[red!100!](.5,-2) circle (.07);
		\node (A4) at (1,-4.5)  {$\bf{B_1:} \text{ rank}=4$};
		\node (A4) at (1,-5)  {$\bf{\text{paw}}$};
		\draw[line width=.2 mm] (4,-2) -- (4.5,-3);
		\draw[line width=.2 mm] (5,-4) -- (4.5,-3);
		\draw[line width=.2 mm] (5,-4) -- (5.5,-3);
		\draw[line width=.2 mm] (5.5,-3) -- (6,-2);
		\draw[line width=.2 mm] (4.5,-3) -- (5.5,-3);
		\fill[red!100!](4,-2) circle (.07);
		\fill[black!100!](5,-4) circle (.07);
		\fill[red!100!](4.5,-3) circle (.07);
		\fill[black!100!](5.5,-3) circle (.07);
		\fill[black!100!](6,-2) circle (.07);
		\node (A4) at (5,-4.5)  {$\bf{B_2:} \text{ rank}=4$};
		\node (A4) at (5,-5)  {$\bf{\text{ bull }}$};
		\draw[line width=.2 mm] (9.5,-2.5) -- (9.5,-1.5);
		\draw[line width=.2 mm] (9.5,-2.5) -- (9,-3);
		\draw[line width=.2 mm] (9.5,-2.5) -- (10,-3);
		\draw[line width=.2 mm] (9,-4) -- (10,-4);
		\draw[line width=.2 mm] (9,-4) -- (9,-3);
		\draw[line width=.2 mm] (10,-4) -- (10,-3);
		\draw[line width=.2 mm] (9,-3) -- (10,-3);
		\fill[red!100!](9.5,-1.5) circle (.07);
		\fill[red!100!](9.5,-2.5) circle (.07);
		\fill[black!100!](9,-4) circle (.07);
		\fill[black!100!](9,-3) circle (.07);
		\fill[black!100!](10,-3) circle (.07);
		\fill[black!100!](10,-4) circle (.07);
		\node (A4) at (9,-4.5)  {$\bf{B_3:} \text{ rank}=4$};
		\node (A4) at (9,-5)  {$\bf{\text{antenna}} $};
		\draw[line width=.2 mm] (9.5,-6.5) -- (9,-7);
		\draw[line width=.2 mm] (9.5,-6.5) -- (10,-7);
		\draw[line width=.2 mm] (9,-8) -- (10,-8);
		\draw[line width=.2 mm] (9,-8) -- (9,-7);
		\draw[line width=.2 mm] (10,-8) -- (10,-7);
		\draw[line width=.2 mm] (9,-7) -- (10,-7);
		\fill[black!100!](9.5,-6.5) circle (.07);
		\fill[red!100!](9,-8) circle (.07);
		\fill[black!100!](9,-7) circle (.07);
		\fill[black!100!](10,-7) circle (.07);
		\fill[red!100!](10,-8) circle (.07);
		\node (A4) at (9.3,-8.5)  {$\bf{B_5:} \text{ rank}=4$};
		\node (A4) at (9,-9)  {$\bf{\text{house}} $};
		\draw[line width=.2 mm] (0,-8) -- (2,-8);
		\draw[line width=.2 mm] (0,-8) -- (0,-6.5);
		\draw[line width=.2 mm] (2,-8) -- (2,-6.5);
		\draw[line width=.2 mm] (0,-6.5) -- (2,-6.5);
		\draw[line width=.2 mm] (0,-6.5) -- (2,-8);
		\draw[line width=.2 mm] (0,-8) -- (2,-6.5);
		\fill[black!100!](0,-8) circle (.07);
		\fill[black!100!](2,-8) circle (.07);
		\fill[black!100!](2,-6.5) circle (.07);
		\fill[black!100!](0,-6.5) circle (.07);
		\node (A4) at (1,-8.5)  {$\bf{K_4:} \text{ rank}=4$};
		\draw[line width=.2 mm] (5,-8) -- (7,-8);
		\draw[line width=.2 mm] (5,-8) -- (5,-6.5);
		\draw[line width=.2 mm] (7,-8) -- (7,-6.5);
		\draw[line width=.2 mm] (5,-6.5) -- (7,-6.5);
		\draw[line width=.2 mm] (5.5,-7.2) -- (5,-8);
		\draw[line width=.2 mm] (5.5,-7.2) -- (5,-6.5);
		\draw[line width=.2 mm] (5.5,-7.2) -- (6.5,-7.2);
		\draw[line width=.2 mm] (7,-8) -- (6.5,-7.2);
		\draw[line width=.2 mm] (7,-6.5) -- (6.5,-7.2);
		\fill[black!100!](5.5,-7.2) circle (.07);
		\fill[black!100!](6.5,-7.2) circle (.07);
		\fill[red!100!](5,-8) circle (.07);
		\fill[red!100!](7,-8) circle (.07);
		\fill[black!100!](7,-6.5) circle (.07);
		\fill[black!100!](5,-6.5) circle (.07);
		\node (A4) at (6.5,-8.5)  {$\bf{B_4:} \text{ rank}=4$};
		\node (A4) at (6.5,-9)  {$\bf{\text{co}-C_6} $};
		\end{tikzpicture}
		\caption{Reduced graphs of rank $4$}
		\label{fig:rank 4 graphs}	
	\end{figure}
With the notation and terminology introduced above we can state the characterization of graphs $\Gamma$ having $r(\Gamma)\leq 5.$ First, we state the result on the rank of paths and cycles.   
\begin{lemma}[\cite{book-spectra-doob}]\label{Lemm: rank of path and cycle}
	Let $P_n$ and $C_n$ be a path and a cycle on $n$ vertices, respectively. Then
	\begin{equation*}
	r(P_n)=	
	\begin{cases}
	n, & \text{ if } n \text{ is even } \\
	n-1,& \text{ if } n \text{ is odd }
	\end{cases}	
	\end{equation*}
	and
	\begin{equation*}
	r(C_n)=	
	\begin{cases}
	n-2, & \text{ if } n \text{ is a  multiple of  }  4 \\
	n,& \text{ otherwise }
	\end{cases}	
	\end{equation*}
\end{lemma} 
\begin{theorem}[\cite{Nulity-of-graphs-ELA,Nulity-bicyclic-graph,Rank-of-graphs-Sc}]\label{thm: cha of graphs rank leq 3}
	Let $\Gamma$ be a simple connected graph, then 
	\begin{enumerate}
		\item $r(\Gamma)=2$ if and only if $\Gamma\in \mathcal{M}(K_2)$
		\item $r(\Gamma)=3$ if and only if $\Gamma\in \mathcal{M}(K_3).$
	\end{enumerate}	
\end{theorem}
As an immediate consequence of Theorem \ref{thm: cha of graphs rank leq 3} and Remark \ref{rem: diameter of graph after multi of comple graph}, we have the following:
\begin{lemma}\label{lem: diam of rank leq 3 graphs}
	Let $\Gamma$ be a simple connected graph such that $r(\Gamma)\leq 3.$ Then $\text{diam}(\Gamma)\leq 2$ if and only if $\Gamma\in \mathcal{M}(K_2, K_3).$ 
\end{lemma} 	
	Now, we present the main theorem in \cite{Rank-4-graph-LAA} that identified all graphs having rank $4.$ 
	\begin{theorem}[\cite{Rank-4-graph-LAA}]\label{thm: cha of graphs of rank 4}
		Let $\Gamma$ be a simple connected graph. Then $r(\Gamma)=4$ if and only if $\Gamma\in \mathcal{M}(B_1, B_2, B_3, B_4, B_5, K_4, P_4, P_5),$ where the graphs $B_1, B_2, B_3, B_4, B_5, K_4, P_4, P_5$ are depicted in Figure \ref{fig:rank 4 graphs}.
	\end{theorem}
Here we classify all graphs $\Gamma$ such that $r(\Gamma)=4$ and $\text{diam}(\Gamma)\leq 3.$
 \begin{lemma}\label{lem: char of graph of rank 4 with diam 2 and 3}
Let $\Gamma$ be a simple connected graph with $r(\Gamma)=4.$ Then $\text{diam}(\Gamma)\leq 3$ if and only if $\Gamma\in \mathcal{M}(B_1, B_2, B_3, B_4, B_5, P_4, K_4).$
	\end{lemma} 
\begin{proof}
Look at the graphs in Figure \ref{fig:rank 4 graphs}. Clearly, the diameter of $P_5$ is $4$ but the diameter of all other non-complete graphs is either $2$ or $3.$ So,  applying Lemma \ref{lem: diam fixed by multiplication vertices}, Remark \ref{rem: diameter of graph after multi of comple graph} and Theorem \ref{thm: cha of graphs of rank 4}, we get this lemma.
\end{proof}
	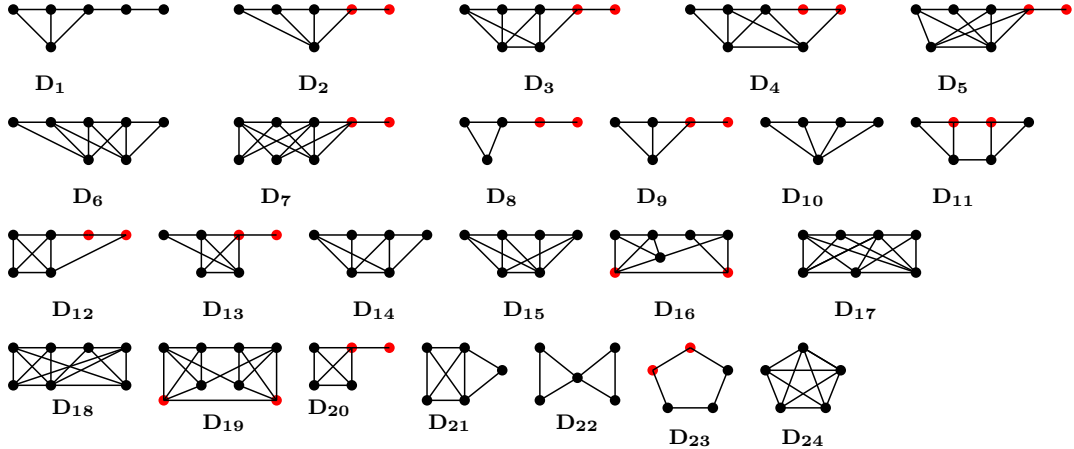
\begin{figure}
		\tiny
		\tikzstyle{ver}=[]
		\tikzstyle{vert}=[circle, draw, fill=black!100, inner sep=0pt, minimum width=4pt]
		\tikzstyle{vertex}=[circle, draw, fill=black!00, inner sep=0pt, minimum width=4pt]
		\tikzstyle{edge} = [draw,thick,-]
		\tikzstyle{node_style} = [circle,draw=blue,fill=blue!20!,font=\sffamily\Large\bfseries]
		\centering
		\begin{tikzpicture}[scale=1]
		\tikzstyle{edge_style} = [draw=black, line width=2mm, ]
		\tikzstyle{node_style} = [draw=blue,fill=blue!00!,font=\sffamily\Large\bfseries]
		\fill[black!100!](0,0) circle (.07);
		\fill[black!100!](.5,0) circle (.07);
		\fill[black!100!](1,0) circle (.07);
		\fill[black!100!](1.5,0) circle (.07);
		\fill[black!100!](2,0) circle (.07);
		\fill[black!100!](.5,-.5) circle (.07);
		\draw[line width=.2 mm] (0,0) -- (.5,0);
		\draw[line width=.2 mm] (.5,0) -- (1.5,0);
		\draw[line width=.2 mm] (2,0) -- (1.5,0);
		\draw[line width=.2 mm] (.5,-.5) -- (0,0);
		\draw[line width=.2 mm] (.5,-.5) -- (.5,0);
		\draw[line width=.2 mm] (.5,-.5) -- (1,0);
		\node (A4) at (.5,-1)  {$\bf{D_1}$};
		\fill[black!100!](3,0) circle (.07);
		\fill[black!100!](3.5,0) circle (.07);
		\fill[black!100!](4,0) circle (.07);
		\fill[red!100!](4.5,0) circle (.07);
		\fill[red!100!](5,0) circle (.07);
		\fill[black!100!](4,-.5) circle (.07);
		\draw[line width=.2 mm] (3,0) -- (3.5,0);
		\draw[line width=.2 mm] (3.5,0) -- (4,0);
		\draw[line width=.2 mm] (4,0) -- (4.5,0);
		\draw[line width=.2 mm] (4.5,0) -- (5,0);
		\draw[line width=.2 mm] (4,-.5) -- (3,0);
		\draw[line width=.2 mm] (4,-.5) -- (3.5,0);
		\draw[line width=.2 mm] (4,-.5) -- (4,0);
		\draw[line width=.2 mm] (4,-.5) -- (4.5,0);
		\node (A4) at (4,-1)  {$\bf{D_2}$};
		\fill[black!100!](6,0) circle (.07);
		\fill[black!100!](6.5,0) circle (.07);
		\fill[black!100!](7,0) circle (.07);
		\fill[red!100!](7.5,0) circle (.07);
		\fill[red!100!](8,0) circle (.07);
		\fill[black!100!](6.5,-.5) circle (.07);
		\fill[black!100!](7,-.5) circle (.07);
		\draw[line width=.2 mm] (6.5,-.5) -- (7,-.5);
		\draw[line width=.2 mm] (6,0) -- (6.5,0);
		\draw[line width=.2 mm] (6.5,0) -- (7,0);
		\draw[line width=.2 mm] (7,0) -- (7.5,0);
		\draw[line width=.2 mm] (7.5,0) -- (8,0);
		\draw[line width=.2 mm] (6.5,-.5) -- (6,0);
		\draw[line width=.2 mm] (6.5,-.5) -- (6.5,0);
		\draw[line width=.2 mm] (6.5,-.5) -- (7,0);
		\draw[line width=.2 mm] (7,-.5) -- (6,0);
		\draw[line width=.2 mm] (7,-.5) -- (7,0);
		\draw[line width=.2 mm] (7,-.5) -- (7.5,0);
		\node (A4) at (7,-1)  {$\bf{D_3}$};
		\fill[black!100!](9,0) circle (.07);
		\fill[black!100!](9.5,0) circle (.07);
		\fill[black!100!](10,0) circle (.07);
		\fill[red!100!](10.5,0) circle (.07);
		\fill[red!100!](11,0) circle (.07);
		\fill[black!100!](9.5,-.5) circle (.07);
		\fill[black!100!](10.5,-.5) circle (.07);
		\draw[line width=.2 mm] (9.5,-.5) -- (10.5,-.5);
		\draw[line width=.2 mm] (9,0) -- (9.5,0);
		\draw[line width=.2 mm] (9.5,0) -- (10,0);
		\draw[line width=.2 mm] (10,0) -- (10.5,0);
		\draw[line width=.2 mm] (10.5,0) -- (11,0);
		\draw[line width=.2 mm] (9.5,-.5) -- (9,0);
		\draw[line width=.2 mm] (9.5,-.5) -- (9.5,0);
		\draw[line width=.2 mm] (9.5,-.5) -- (10,0);
		\draw[line width=.2 mm] (10.5,-.5) -- (9.5,0);
		\draw[line width=.2 mm] (10.5,-.5) -- (10,0);
		\draw[line width=.2 mm] (10.5,-.5) -- (11,0);
		\node (A4) at (10,-1)  {$\bf{D_4}$};
		\fill[black!100!](12,0) circle (.07);
		\fill[black!100!](12.5,0) circle (.07);
		\fill[black!100!](13,0) circle (.07);
		\fill[red!100!](13.5,0) circle (.07);
		\fill[red!100!](14,0) circle (.07);
		\fill[black!100!](12.2,-.5) circle (.07);
		\fill[black!100!](13,-.5) circle (.07);
		\draw[line width=.2 mm] (12.2,-.5) -- (13,-.5);
		\draw[line width=.2 mm] (12,0) -- (12.5,0);
		\draw[line width=.2 mm] (12.5,0) -- (13,0);
		\draw[line width=.2 mm] (13,0) -- (13.5,0);
		\draw[line width=.2 mm] (13.5,0) -- (14,0);
		\draw[line width=.2 mm] (12.2,-.5) -- (12,0);
		\draw[line width=.2 mm] (12.2,-.5) -- (13,0);
		\draw[line width=.2 mm] (12.2,-.5) -- (13.5,0);
		\draw[line width=.2 mm] (13,-.5) -- (12,0);
		\draw[line width=.2 mm] (13,-.5) -- (12.5,0);
		\draw[line width=.2 mm] (13,-.5) -- (13,0);
		\draw[line width=.2 mm] (13,-.5) -- (13.5,0);
		\node (A4) at (12.5,-1)  {$\bf{D_5}$};
		\fill[black!100!](0,-1.5) circle (.07);
		\fill[black!100!](.5,-1.5) circle (.07);
		\fill[black!100!](1,-1.5) circle (.07);
		\fill[black!100!](1.5,-1.5) circle (.07);
		\fill[black!100!](2,-1.5) circle (.07);
		\fill[black!100!](1,-2) circle (.07);
		\fill[black!100!](1.5,-2) circle (.07);
		\draw[line width=.2 mm] (0,-1.5) -- (.5,-1.5);
		\draw[line width=.2 mm] (.5,-1.5) -- (1.5,-1.5);
		\draw[line width=.2 mm] (2,-1.5) -- (1.5,-1.5);
		\draw[line width=.2 mm] (1,-2) -- (0,-1.5);
		\draw[line width=.2 mm] (1,-2) -- (.5,-1.5);
		\draw[line width=.2 mm] (1,-2) -- (1,-1.5);
		\draw[line width=.2 mm] (1,-2) -- (1.5,-1.5);
		\draw[line width=.2 mm] (1.5,-2) -- (.5,-1.5);
		\draw[line width=.2 mm] (1.5,-2) -- (1,-1.5);
		\draw[line width=.2 mm] (1.5,-2) -- (1.5,-1.5);
		\draw[line width=.2 mm] (1.5,-2) -- (2,-1.5);
		\node (A4) at (1,-2.5)  {$\bf{D_6}$};
		\fill[black!100!](3,-1.5) circle (.07);
		\fill[black!100!](3.5,-1.5) circle (.07);
		\fill[black!100!](4,-1.5) circle (.07);
		\fill[red!100!](4.5,-1.5) circle (.07);
		\fill[red!100!](5,-1.5) circle (.07);
		\fill[black!100!](3,-2) circle (.07);
		\fill[black!100!](3.5,-2) circle (.07);
		\fill[black!100!](4,-2) circle (.07);
		\draw[line width=.2 mm] (3,-1.5) -- (3.5,-1.5);
		\draw[line width=.2 mm] (3.5,-1.5) -- (4,-1.5);
		\draw[line width=.2 mm] (4,-1.5) -- (4.5,-1.5);
		\draw[line width=.2 mm] (4.5,-1.5) -- (5,-1.5);
		\draw[line width=.2 mm] (3,-2) -- (3,-1.5);
		\draw[line width=.2 mm] (3,-2) -- (3.5,-1.5);
		\draw[line width=.2 mm] (3,-2) -- (4,-1.5);
		\draw[line width=.2 mm] (3.5,-2) -- (3,-1.5);
		\draw[line width=.2 mm] (3.5,-2) -- (4,-1.5);
		\draw[line width=.2 mm] (3.5,-2) --(4.5,-1.5);
		\draw[line width=.2 mm] (4,-2) -- (3,-1.5);
		\draw[line width=.2 mm] (4,-2) -- (3.5,-1.5);
		\draw[line width=.2 mm] (4,-2) -- (4,-1.5);
		\draw[line width=.2 mm] (4,-2) -- (4.5,-1.5);
		\node (A4) at (3.5,-2.5)  {$\bf{D_7}$};
		\fill[black!100!](6,-1.5) circle (.07);
		\fill[black!100!](6.5,-1.5) circle (.07);
		\fill[red!100!](7,-1.5) circle (.07);
		\fill[red!100!](7.5,-1.5) circle (.07);
		\fill[black!100!](6.3,-2) circle (.07);
		%
		%
		\draw[line width=.2 mm] (6,-1.5) -- (6.5,-1.5);
		\draw[line width=.2 mm] (6.5,-1.5) -- (7,-1.5);
		\draw[line width=.2 mm] (7,-1.5) -- (7.5,-1.5);
		\draw[line width=.2 mm] (6.3,-2) -- (6,-1.5);
		\draw[line width=.2 mm] (6.3,-2) -- (6.5,-1.5);
		\node (A4) at (6.5,-2.5)  {$\bf{D_8}$};
		\fill[black!100!](8,-1.5) circle (.07);
		\fill[black!100!](8.5,-1.5) circle (.07);
		\fill[red!100!](9,-1.5) circle (.07);
		\fill[red!100!](9.5,-1.5) circle (.07);
		\fill[black!100!](8.5,-2) circle (.07);
		%
		%
		\draw[line width=.2 mm] (8,-1.5) -- (8.5,-1.5);
		\draw[line width=.2 mm] (8.5,-1.5) -- (9,-1.5);
		\draw[line width=.2 mm] (9,-1.5) -- (9.5,-1.5);
		\draw[line width=.2 mm] (8.5,-2) -- (8,-1.5);
		\draw[line width=.2 mm] (8.5,-2) -- (8.5,-1.5);
		\draw[line width=.2 mm] (8.5,-2) -- (9,-1.5);
		\node (A4) at (8.5,-2.5)  {$\bf{D_9}$};
		\fill[black!100!](10,-1.5) circle (.07);
		\fill[black!100!](10.5,-1.5) circle (.07);
		\fill[black!100!](11,-1.5) circle (.07);
		\fill[black!100!](11.5,-1.5) circle (.07);
		\fill[black!100!](10.7,-2) circle (.07);
		%
		%
		\draw[line width=.2 mm] (10,-1.5) -- (10.5,-1.5);
		\draw[line width=.2 mm] (10.5,-1.5) -- (11,-1.5);
		\draw[line width=.2 mm] (11,-1.5) -- (11.5,-1.5);
		\draw[line width=.2 mm] (10.7,-2) -- (10,-1.5);
		\draw[line width=.2 mm] (10.7,-2) -- (10.5,-1.5);
		\draw[line width=.2 mm] (10.7,-2) -- (11,-1.5);
		\draw[line width=.2 mm] (10.7,-2) -- (11.5,-1.5);
		\node (A4) at (10.5,-2.5)  {$\bf{D_{10}}$};
		\fill[black!100!](12,-1.5) circle (.07);
		\fill[red!100!](12.5,-1.5) circle (.07);
		\fill[red!100!](13,-1.5) circle (.07);
		\fill[black!100!](13.5,-1.5) circle (.07);
		\fill[black!100!](12.5,-2) circle (.07);
		\fill[black!100!](13,-2) circle (.07);
		\draw[line width=.2 mm] (12.5,-2) -- (13,-2);
		\draw[line width=.2 mm] (12,-1.5) -- (12.5,-1.5);
		\draw[line width=.2 mm] (12.5,-1.5) -- (13,-1.5);
		\draw[line width=.2 mm] (13,-1.5) -- (13.5,-1.5);
		\draw[line width=.2 mm] (12.5,-2) -- (12,-1.5);
		\draw[line width=.2 mm] (12.5,-2) -- (12.5,-1.5);
		\draw[line width=.2 mm] (13,-2) -- (13,-1.5);
		\draw[line width=.2 mm] (13,-2) -- (13.5,-1.5);
		\node (A4) at (12.5,-2.5)  {$\bf{D_{11}}$};
		\fill[black!100!](0,-3) circle (.07);
		\fill[black!100!](.5,-3) circle (.07);
		\fill[red!100!](1,-3) circle (.07);
		\fill[red!100!](1.5,-3) circle (.07);
		\fill[black!100!](0,-3.5) circle (.07);
		\fill[black!100!](.5,-3.5) circle (.07);
		\draw[line width=.2 mm] (0,-3) -- (.5,-3);
		\draw[line width=.2 mm] (.5,-3) -- (1,-3);
		\draw[line width=.2 mm] (1,-3) -- (1.5,-3);
		\draw[line width=.2 mm] (0,-3.5) -- (0,-3);
		\draw[line width=.2 mm] (0,-3.5) -- (.5,-3);
		\draw[line width=.2 mm] (0,-3.5) -- (.5,-3.5);
		\draw[line width=.2 mm] (.5,-3.5) -- (1.5,-3);
		\draw[line width=.2 mm] (.5,-3.5) -- (.5,-3);
		\draw[line width=.2 mm] (.5,-3.5) -- (0,-3);
		\node (A4) at (.8,-4)  {$\bf{D_{12}}$};
		\fill[black!100!](2,-3) circle (.07);
		\fill[black!100!](2.5,-3) circle (.07);
		\fill[red!100!](3,-3) circle (.07);
		\fill[red!100!](3.5,-3) circle (.07);
		\fill[black!100!](2.5,-3.5) circle (.07);
		\fill[black!100!](3,-3.5) circle (.07);
		\draw[line width=.2 mm] (2,-3) -- (2.5,-3);
		\draw[line width=.2 mm] (2.5,-3) -- (3,-3);
		\draw[line width=.2 mm] (3,-3) -- (3.5,-3);
		\draw[line width=.2 mm] (2.5,-3.5) -- (2.5,-3);
		\draw[line width=.2 mm] (2.5,-3.5) -- (3,-3);
		\draw[line width=.2 mm] (2.5,-3.5) -- (3,-3.5);
		\draw[line width=.2 mm] (3,-3.5) -- (3,-3);
		\draw[line width=.2 mm] (3,-3.5) -- (2.5,-3);
		\draw[line width=.2 mm] (3,-3.5) -- (2,-3);
		\node (A4) at (2.8,-4)  {$\bf{D_{13}}$};
		\fill[black!100!](4,-3) circle (.07);
		\fill[black!100!](4.5,-3) circle (.07);
		\fill[black!100!](5,-3) circle (.07);
		\fill[black!100!](5.5,-3) circle (.07);
		\fill[black!100!](4.5,-3.5) circle (.07);
		\fill[black!100!](5,-3.5) circle (.07);
		\draw[line width=.2 mm] (4,-3) -- (4.5,-3);
		\draw[line width=.2 mm] (4.5,-3) -- (5,-3);
		\draw[line width=.2 mm] (5,-3) -- (5.5,-3);
		\draw[line width=.2 mm] (4.5,-3.5) -- (4,-3);
		\draw[line width=.2 mm] (4.5,-3.5) -- (4.5,-3);
		\draw[line width=.2 mm] (4.5,-3.5) -- (5,-3);
		\draw[line width=.2 mm] (4.5,-3.5) -- (5,-3.5);
		\draw[line width=.2 mm] (5,-3.5) -- (5,-3);
		\draw[line width=.2 mm] (5,-3.5) -- (4,-3);
		\draw[line width=.2 mm] (5,-3.5) -- (5.5,-3);
		\node (A4) at (4.8,-4)  {$\bf{D_{14}}$};
		\fill[black!100!](6,-3) circle (.07);
		\fill[black!100!](6.5,-3) circle (.07);
		\fill[black!100!](7,-3) circle (.07);
		\fill[black!100!](7.5,-3) circle (.07);
		\fill[black!100!](6.5,-3.5) circle (.07);
		\fill[black!100!](7,-3.5) circle (.07);
		\draw[line width=.2 mm] (6,-3) -- (6.5,-3);
		\draw[line width=.2 mm] (6.5,-3) -- (7,-3);
		\draw[line width=.2 mm] (7.5,-3) -- (7,-3);
		\draw[line width=.2 mm] (6.5,-3.5) -- (6,-3);
		\draw[line width=.2 mm] (6.5,-3.5) -- (6.5,-3);
		\draw[line width=.2 mm] (6.5,-3.5) -- (7.5,-3);
		\draw[line width=.2 mm] (6.5,-3.5) -- (7,-3.5);
		\draw[line width=.2 mm] (7,-3.5) -- (6,-3);
		\draw[line width=.2 mm] (7,-3.5) -- (6.5,-3);
		\draw[line width=.2 mm] (7,-3.5) -- (7,-3);
		\draw[line width=.2 mm] (7,-3.5) -- (7.5,-3);
		\node (A4) at (6.8,-4)  {$\bf{D_{15}}$};
		\fill[black!100!](8,-3) circle (.07);
		\fill[black!100!](8.5,-3) circle (.07);
		\fill[black!100!](9,-3) circle (.07);
		\fill[black!100!](9.5,-3) circle (.07);
		\fill[red!100!](8,-3.5) circle (.07);
		\fill[red!100!](9.5,-3.5) circle (.07);
		\fill[black!100!](8.6,-3.3) circle (.07);
		\draw[line width=.2 mm] (8,-3) -- (8.5,-3);
		\draw[line width=.2 mm] (8.5,-3) -- (9,-3);
		\draw[line width=.2 mm] (9.5,-3) -- (9,-3);
		\draw[line width=.2 mm] (8,-3.5) -- (9.5,-3.5);
		\draw[line width=.2 mm] (8,-3.5) -- (8,-3);
		\draw[line width=.2 mm] (8,-3.5) -- (8.5,-3);
		\draw[line width=.2 mm] (8,-3.5) -- (8.6,-3.3);
		\draw[line width=.2 mm] (9.5,-3.5) -- (9.5,-3);
		\draw[line width=.2 mm] (9.5,-3.5) -- (9,-3);
		\draw[line width=.2 mm] (8.6,-3.3) -- (9.5,-3);
		\draw[line width=.2 mm] (8.6,-3.3) -- (8.5,-3);
		\draw[line width=.2 mm] (8.6,-3.3) -- (8,-3);
		\node (A4) at (8.8,-4)  {$\bf{D_{16}}$};
		\fill[black!100!](10.5,-3) circle (.07);
		\fill[black!100!](11,-3) circle (.07);
		\fill[black!100!](11.5,-3) circle (.07);
		\fill[black!100!](12,-3) circle (.07);
		\fill[black!100!](10.5,-3.5) circle (.07);
		\fill[black!100!](11.2,-3.5) circle (.07);
		\fill[black!100!](12,-3.5) circle (.07);
		\draw[line width=.2 mm] (10.5,-3) -- (11,-3);
		\draw[line width=.2 mm] (11,-3) -- (11.5,-3);
		\draw[line width=.2 mm] (12,-3) -- (11.5,-3);
		\draw[line width=.2 mm] (11,-3) -- (10.5,-3.5);
		\draw[line width=.2 mm] (10.5,-3.5) -- (11.2,-3.5);
		\draw[line width=.2 mm] (10.5,-3.5) -- (10.5,-3);
		\draw[line width=.2 mm] (10.5,-3.5) -- (11.5,-3);
		\draw[line width=.2 mm] (10.5,-3.5) -- (11.5,-3);
		\draw[line width=.2 mm] (11.2,-3.5) -- (12,-3.5);
		\draw[line width=.2 mm] (11.2,-3.5) -- (10.5,-3);
		\draw[line width=.2 mm] (11.2,-3.5) -- (11.5,-3);
		\draw[line width=.2 mm] (11.2,-3.5) -- (12,-3);
		\draw[line width=.2 mm] (12,-3.5) -- (12,-3);
		\draw[line width=.2 mm] (12,-3.5) -- (11.5,-3);
		\draw[line width=.2 mm] (12,-3.5) -- (11,-3);
		\draw[line width=.2 mm] (12,-3.5) -- (10.5,-3);
		\node (A4) at (11.2,-4)  {$\bf{D_{17}}$};
		\fill[black!100!](0,-4.5) circle (.07);
		\fill[black!100!](.5,-4.5) circle (.07);
		\fill[black!100!](1,-4.5) circle (.07);
		\fill[black!100!](1.5,-4.5) circle (.07);
		\fill[black!100!](0,-5) circle (.07);
		\fill[black!100!](.5,-5) circle (.07);
		\fill[black!100!](1.5,-5) circle (.07);
		\draw[line width=.2 mm] (0,-4.5) -- (.5,-4.5);
		\draw[line width=.2 mm] (.5,-4.5) -- (1,-4.5);
		\draw[line width=.2 mm] (1,-4.5) -- (1.5,-4.5);
		\draw[line width=.2 mm] (0,-5) -- (0,-4.5);
		\draw[line width=.2 mm] (0,-5) -- (.5,-4.5);
		\draw[line width=.2 mm] (0,-5) -- (1.5,-4.5);
		\draw[line width=.2 mm] (0,-5) -- (.5,-5);
		\draw[line width=.2 mm] (.5,-5) -- (0,-4.5);
		\draw[line width=.2 mm] (.5,-5) -- (.5,-4.5);
		\draw[line width=.2 mm] (.5,-5) -- (1,-4.5);
		\draw[line width=.2 mm] (.5,-5) -- (1.5,-4.5);
		\draw[line width=.2 mm] (.5,-5) -- (1.5,-5);
		\draw[line width=.2 mm] (1.5,-5) -- (0,-4.5);
		\draw[line width=.2 mm] (1.5,-5) -- (1,-4.5);
		\draw[line width=.2 mm] (1.5,-5) -- (1.5,-4.5);
		\node (A4) at (.8,-5.3)  {$\bf{D_{18}}$};
		\fill[black!100!](2,-4.5) circle (.07);
		\fill[black!100!](2.5,-4.5) circle (.07);
		\fill[black!100!](3,-4.5) circle (.07);
		\fill[black!100!](3.5,-4.5) circle (.07);
		\fill[red!100!](2,-5.2) circle (.07);
		\fill[red!100!](3.5,-5.2) circle (.07);
		\fill[black!100!](2.5,-5) circle (.07);
		\fill[black!100!](3,-5) circle (.07);
		\draw[line width=.2 mm] (2,-4.5) -- (2.5,-4.5);
		\draw[line width=.2 mm] (2.5,-4.5) -- (3,-4.5);
		\draw[line width=.2 mm] (3,-4.5) -- (3.5,-4.5);
		\draw[line width=.2 mm] (2,-5.2) -- (2,-4.5);
		\draw[line width=.2 mm] (2,-5.2) -- (3.5,-5.2);
		\draw[line width=.2 mm] (2,-5.2) -- (2.5,-4.5);
		\draw[line width=.2 mm] (2,-5.2) -- (2.5,-5);
		\draw[line width=.2 mm] (3.5,-5.2) -- (3.5,-4.5);
		\draw[line width=.2 mm] (2.5,-5) -- (3.5,-4.5);
		\draw[line width=.2 mm] (3.5,-5.2) -- (3,-4.5);
		\draw[line width=.2 mm] (3.5,-5.2) -- (3,-5);
		\draw[line width=.2 mm] (2.5,-5) -- (2.5,-4.5);
		\draw[line width=.2 mm] (2.5,-5) -- (2,-4.5);
		\draw[line width=.2 mm] (3,-5) -- (3,-4.5);
		\draw[line width=.2 mm] (3,-5) -- (2,-4.5);
		\draw[line width=.2 mm] (3,-5) -- (3.5,-4.5);
		\node (A4) at (2.8,-5.5)  {$\bf{D_{19}}$};
		\fill[black!100!](4,-4.5) circle (.07);
		\fill[red!100!](4.5,-4.5) circle (.07);
		\fill[red!100!](5,-4.5) circle (.07);
		\fill[black!100!](4,-5) circle (.07);
		\fill[black!100!](4.5,-5) circle (.07);
		\draw[line width=.2 mm] (4,-4.5) -- (4.5,-4.5);
		\draw[line width=.2 mm] (4.5,-4.5) -- (5,-4.5);
		\draw[line width=.2 mm] (4,-5) -- (4,-4.5);
		\draw[line width=.2 mm] (4,-5) -- (4.5,-5);
		\draw[line width=.2 mm] (4,-5) -- (4.5,-4.5);
		\draw[line width=.2 mm] (4.5,-5) -- (4.5,-4.5);
		\draw[line width=.2 mm] (4.5,-5) -- (4,-4.5);
		\node (A4) at (4.2,-5.3)  {$\bf{D_{20}}$};
		\fill[black!100!](5.5,-4.5) circle (.07);
		\fill[black!100!](6,-4.5) circle (.07);
		\fill[black!100!](6.5,-4.8) circle (.07);
		\fill[black!100!](5.5,-5.2) circle (.07);
		\fill[black!100!](6,-5.2) circle (.07);
		\draw[line width=.2 mm] (5.5,-4.5) -- (6,-4.5);
		\draw[line width=.2 mm] (5.5,-4.5) -- (5.5,-5.2);
		\draw[line width=.2 mm] (5.5,-4.5) -- (6,-5.2);
		\draw[line width=.2 mm] (6,-4.5) -- (6,-5.2);
		\draw[line width=.2 mm] (6,-4.5) -- (5.5,-5.2);
		\draw[line width=.2 mm] (6.5,-4.8) -- (6,-4.5);
		\draw[line width=.2 mm] (6.5,-4.8) -- (6,-5.2);
		\draw[line width=.2 mm] (5.5,-5.2) -- (6,-5.2);
		\node (A4) at (5.8,-5.5)  {$\bf{D_{21}}$};
		\fill[black!100!](7,-4.5) circle (.07);
		\fill[black!100!](8,-4.5) circle (.07);
		\fill[black!100!](7,-5.2) circle (.07);
		\fill[black!100!](8,-5.2) circle (.07);
		\fill[black!100!](7.5,-4.9) circle (.07);
		\draw[line width=.2 mm] (7,-4.5) -- (7.5,-4.9);
		\draw[line width=.2 mm] (7,-4.5) -- (7,-5.2);
		\draw[line width=.2 mm] (8,-4.5) -- (7.5,-4.9);
		\draw[line width=.2 mm] (8,-4.5) -- (8,-5.2);
		\draw[line width=.2 mm] (7,-5.2) -- (7.5,-4.9);
		\draw[line width=.2 mm] (8,-5.2) -- (7.5,-4.9);
		\node (A4) at (7.5,-5.5)  {$\bf{D_{22}}$};
		\fill[red!100!](9,-4.5) circle (.07);
		\fill[red!100!](8.5,-4.8) circle (.07);
		\fill[black!100!](9.5,-4.8) circle (.07);
		\fill[black!100!](8.7,-5.3) circle (.07);
		\fill[black!100!](9.3,-5.3) circle (.07);
		\draw[line width=.2 mm] (9,-4.5) -- (8.5,-4.8);
		\draw[line width=.2 mm] (9,-4.5) -- (9.5,-4.8);
		\draw[line width=.2 mm] (8.5,-4.8) -- (8.7,-5.3);
		\draw[line width=.2 mm] (8.7,-5.3) -- (9.3,-5.3);
		\draw[line width=.2 mm] (9.3,-5.3) -- (9.5,-4.8);
		\node (A4) at (9,-5.7)  {$\bf{D_{23}}$};
		\fill[black!100!](10.5,-4.5) circle (.07);
		\fill[black!100!](10,-4.8) circle (.07);
		\fill[black!100!](11,-4.8) circle (.07);
		\fill[black!100!](10.2,-5.3) circle (.07);
		\fill[black!100!](10.8,-5.3) circle (.07);
		\draw[line width=.2 mm] (10.5,-4.5) -- (11,-4.8);
		\draw[line width=.2 mm] (10.5,-4.5) -- (11,-4.8);
		\draw[line width=.2 mm] (10,-4.8) -- (10.2,-5.3);
		\draw[line width=.2 mm] (10.2,-5.3) -- (10.8,-5.3);
		\draw[line width=.2 mm] (10,-4.8) -- (10.5,-4.5);
		\draw[line width=.2 mm] (10.8,-5.3) -- (11,-4.8);
		\draw[line width=.2 mm] (10.5,-4.5) -- (10.8,-5.3);
		\draw[line width=.2 mm] (10.5,-4.5) -- (10.2,-5.3);
		\draw[line width=.2 mm] (10,-4.8) -- (10.8,-5.3);
		\draw[line width=.2 mm] (10,-4.8) -- (11,-4.8);
		\draw[line width=.2 mm] (10.2,-5.3) -- (11,-4.8);
		\node (A4) at (10.5,-5.7)  {$\bf{D_{24}}$};
		\end{tikzpicture}
		\caption{Reduced graphs of rank $5$}
		\label{fig:rank 5 graphs}	
	\end{figure}
	Now, we state the main theorem in \cite{LAA-rank-5-graph-charecterization} that characterized all graphs of rank $5.$
	\begin{theorem}[\cite{LAA-rank-5-graph-charecterization}]\label{thm: cha of graphs of rank 5}
		Let $\Gamma$ be a simple connected graph. Then $r(\Gamma)=5$ if and only if $\Gamma\in\mathcal{M}(D_1, D_2, \cdots, D_{24}),$ where the graphs $D_1, D_2, \cdots, D_{24}$ are described in Figure \ref{fig:rank 5 graphs}.
	\end{theorem}
	 \begin{lemma}\label{lem: ch of graphs of rank five graphs and diam leq 3}
		Let $\Gamma$ be a simple connected graph such that $r(\Gamma)=5,$ then 	
		$\text{diam}(\Gamma)\leq 3$ if and only if $\Gamma\in \mathcal{M}(D_{2}, D_{3},  \cdots, D_{24})$ 
	\end{lemma}
\begin{proof}
Notice that the graphs in Figure \ref{fig:rank 5 graphs}, clearly the diameter of the graph $D_1$ is $4.$ But the diameter of all other graphs is $\leq 3.$ Therefore, this lemma follows from Lemma \ref{lem: diam fixed by multiplication vertices}, Remark \ref{rem: diameter of graph after multi of comple graph} and Theorem \ref{thm: cha of graphs of rank 5}.
\end{proof}
\section{Main results}\label{Sec: main results}
Our aim in this paper is to prove Conjecture \ref{conj: Cameron on row space}. To do so first we study the graphs having diameter $\geq 4,$ and then we examine the rest of the graphs. For the graphs having diameter $\geq 4,$ we have a theorem (Theorem \ref{thm: proof of conj diam geq 4}) in support of the conjecture. Furthermore, for diameter $\leq 3$ case, we provide two theorems (Theorems \ref{thm: rank 5 graph follows conj}, \ref{thm: dim 2 graph m=2n-5}) in favour of the conjecture.

Now, we state our main results and proofs of these results are given in Section \ref{Sec: proof of main thm}.
\begin{theorem}\label{thm: proof of conj diam geq 4}
Let $\Gamma$ be any simple connected  graph with $\text{diam}(\Gamma)\geq 4.$ Then there exists a non-zero $(0,1)$-vector in the row space of $A(\Gamma)$ (over the real numbers) which does not occur as a row of $A(\Gamma).$
\end{theorem}
Now, we consider all graphs of diameter $\leq 3.$ In this case, we present some progress in support of the conjecture. Note that a graph is of diameter $1$ if and only if it is complete. For the complete graph $K_n,$ the vector $(\underbrace{1, 1, \cdots, 1}_{n \text{ entries}})$ serves the purpose. In fact,
\begin{equation}\label{dis: complete graph satis Conj}
(\underbrace{1, 1, \cdots, 1}_{n \text{ entries}})=\frac{1}{(n-1)}\sum\limits_{i=1}^nR_{i}.
\end{equation}
So, here we focus on the non-complete graph $\Gamma$ such that $\text{diam}(\Gamma)\leq 3.$ In this case, we have following two theorems. 
\begin{theorem}\label{thm: rank 5 graph follows conj}
Let $\Gamma$ be a simple non-complete connected  graph such that $\text{diam}(\Gamma)\leq 3$ and $r(\Gamma)\leq 5.$ Then there exists a non-zero $(0,1)$-vector in the row space of $A(\Gamma)$ (over the real numbers) which does not occur as a row of $A(\Gamma).$	
\end{theorem}
Let $\mathcal{H}$ be the family of graphs defined as in Example \ref{Exam: duplication and family H}. Note that, $C_5, \Gamma_0, \Gamma_5\in \mathcal{H}$ and the diameter of the each of the graphs is $2.$ Also it can be shown that $r(C_5)=5, r(\Gamma_0)=7$ and $r(\Gamma_5)=10.$ Let $\mathcal{D}$ be the collection of all graphs that contains
\begin{enumerate}
\item the graph $\Gamma_0;$ and
\item is closed under degree-$2$ vertex duplication. 
\end{enumerate}
Clearly, $\mathcal{D}\subseteq\mathcal{H}.$ Let $\Gamma\in \mathcal{D},$ then by Proposition \ref{Prop: rank of graph and its multipl aresame} $r(\Gamma)=7,$ (as $r(\Gamma_0)=7).$ Therefore, the family $\mathcal{H}$ contains infinitely many graphs having rank $7.$ The next theorem tells us that each graph of the family $\mathcal{H}$ satisfies the conjecture.
\begin{theorem}\label{thm: dim 2 graph m=2n-5}  
Let $\Gamma$ be a diameter-$2$ graph of order $r$ size $s$ with no dominating vertex. Suppose $s=2r-5.$ Then there exists a non-zero $(0,1)$-vector in the row space of $A(\Gamma)$ (over the real numbers) which does not occur as a row of $A(\Gamma).$
\end{theorem} 
\section{Proof of main results}\label{Sec: proof of main thm}
In this section, we present the proofs of our main theorems.  
\begin{proof}[Proof of Theorem \ref{thm: proof of conj diam geq 4}]
Let $\Gamma$ be a simple connected graph with a diameter of at least 4. Denote the vertex set of $\Gamma$ as $V(\Gamma) = \{v_1, v_2, \ldots, v_n\}$. Given that $\text{diam}(\Gamma) \geq 4$, there exists a shortest path, denoted as $\mathcal{P} = v_1 \sim v_2 \sim v_3 \sim v_4 \sim v_5$, which has a length of 4. Let $R_i$ represent the $i^{\text{th}}$ row of the adjacency matrix $A(\Gamma)$. We assert that the sum $R_2 + R_5$ forms a non-zero $(0, 1)$-vector within the row space of $A(\Gamma)$ (considered over the real numbers) that does not correspond to any row of $A(\Gamma)$. It is evident that
\begin{equation}\label{eq: required vector}
R_{2} + R_{5} = (1, 0, 1, 1, 0, \underbrace{*, \ldots, *}_{(n-5) \text{ entries}}).
\end{equation}
It is important to note that the remaining $(n-5)$ entries of $R_2 + R_5$ can only be either 0 or 1. If not, there would exist a vertex $v \in V(\Gamma)$ such that $v_2 \sim v \sim v_5$, leading to a shortest path $v_1 \sim v_2 \sim v \sim v_5$ of length 3. This would contradict the established distance $d(v_1, v_5) = 4$.

Next, we demonstrate that $R_2 + R_5$ cannot be represented as a row of $A(\Gamma)$. Assume, for contradiction, that there exists an index $i \in [n]$ such that $R_i = R_2 + R_5$. The first and fourth entries of the vector $R_2 + R_5$ are both 1, indicating that the vertex $v_i$ is adjacent to both $v_1$ and $v_5$. Consequently, this would create a path $v_1 \sim v_i \sim v_4\sim v_5.$ This contradicts that $d(v_1, v_5) = 4$. This completes the proof.
\end{proof}
In this portion, we shift our attention to prove the remaining theorems, that is, Theorems \ref{thm: rank 5 graph follows conj}, \ref{thm: dim 2 graph m=2n-5}. To prove these theorems, first we need to prove the following results.
\begin{theorem}\label{thm: v edge u; nbd u meet nbd v=null}
Let $\Gamma$ be any simple connected graph.  Suppose $\Gamma$ has at least two distinct vertices $v_i, v_j$ such that $v_i\sim v_j$ and $\text{nbd}(v_i)\cap \text{nbd}(v_j)=\emptyset.$ Then there exists a non-zero $(0,1)$-vector in the row space of $A(\Gamma)$ (over the real numbers) which does not occur as a row of $A(\Gamma).$
\end{theorem}
\begin{proof}
Let $V(\Gamma)=\{v_1, \cdots, v_n\}.$ Without loss of generality we assume that $v_1\sim v_2$ and $\text{nbd}(v_1)\cap \text{nbd}(v_2)=\emptyset.$ Suppose that $\text{nbd}(v_1)=\{v_{11}, \cdots, v_{1k}\}$ and $\text{nbd}(v_2)=\{v_{21}, \cdots, v_{2r}\}.$ Now, we arrange the vertices of $\Gamma$ in the following order:
\[v_1, v_2, v_{11}, \cdots, v_{1k}, v_{21}, \cdots, v_{2r}, u_1,  \cdots, u_t,\] where $t=n-(k+r+2)$ and $u_i\in V(\Gamma)\setminus\left(\{v_1, v_2\}\cup \text{nbd}(v_1)\cup\text{nbd}(v_2)\right).$  	
Now, it is easy to see that the first row and the second row of $A(\Gamma)$ are
\begin{align*}
R_{1}&=(\underbrace{0,1,\cdots, 1}_{k+2 \text{ entries}},\underbrace{0, \cdots, 0}_{r \text{ entries}}, \underbrace{0,\cdots, 0}_{t \text{ entries}})\\
R_{2}&=(\underbrace{1,0,\cdots, 0}_{k+2 \text{ entries}},\underbrace{1, \cdots, 1}_{r \text{ entries}}, \underbrace{0,\cdots, 0}_{t \text{ entries}}).
\end{align*}
\[\text{ As a result, } R_1+R_2=(\underbrace{1,1,\cdots, 1}_{k+2 \text{ entries}},\underbrace{1, \cdots, 1}_{r \text{ entries}}, \underbrace{0,\cdots, 0}_{t \text{ entries}}).\]
Our claim is that the vector $R_{1}+R_{2}$ serves the purpose. Clearly, $R_1+R_2$ is a non zero $(0,1)$-vector. To complete the proof we have to show that the vector $R_{1}+R_{2}$ does not occur as a row of $A(\Gamma).$ Suppose there exists $i\in [n]$ such that $R_i=R_{1}+R_{2}.$ Then $v_i$ is adjacent to both of the vertices $v_1$ and $v_2$ (since the first two entries of $R_{1}+R_{2}$ are $1).$ That is, $v_i\in\text{nbd}(v_1)\cap \text{nbd}(v_2).$ This contradicts the fact $\text{nbd}(v_1)\cap \text{nbd}(v_2)=\emptyset.$ Hence the theorem.
\end{proof}
\begin{remark}
Note that Theorem \ref{thm: v edge u; nbd u meet nbd v=null} does not depend on the diameter of graph. 	
\end{remark}
\begin{corollary}\label{cor: graph has one pendent vertex, conj hold}
Let $\Gamma$ be a simple connected graph having a vertex $v\in V(\Gamma)$ such that $\text{deg}(v)=1.$ Then $A(\Gamma)$ satisfies Conjecture \ref{conj: Cameron on row space}.	
\end{corollary} 
\begin{proof}
Suppose $\text{deg}(v)=1$ and $v\sim u.$ Then $\text{nbd}(u)\cap \text{nbd}(v)=\emptyset.$	
\end{proof}
\begin{corollary}\label{cor: tree satisfies conj} 
Let $T$ be a tree. Then $A(T)$ satisfies Conjecture \ref{conj: Cameron on row space}.
\end{corollary}
Let $\Gamma$ be a simple connected graph such that $\text{diam}(\Gamma)= 3.$ Then there exist at least two vertices $u, v\in V(\Gamma)$ such that $nbd(u)\cap nbd(v)=\emptyset.$  Otherwise, $\text{diam}(\Gamma)$ would be equal to $2.$ Suppose $u\sim v$ in $\Gamma.$ Then by Theorem \ref{thm: v edge u; nbd u meet nbd v=null}, we can say that $A(\Gamma)$ satisfies the conjecture. In Example \ref{exm: diam 3}, we construct a family of graphs of diameter $3$ that satisfy the conjecture.
 \begin{figure}
		\tiny
		\tikzstyle{ver}=[]
		\tikzstyle{vert}=[circle, draw, fill=black!100, inner sep=0pt, minimum width=4pt]
		\tikzstyle{vertex}=[circle, draw, fill=black!00, inner sep=0pt, minimum width=4pt]
		\tikzstyle{edge} = [draw,thick,-]
		\tikzstyle{node_style} = [circle,draw=blue,fill=blue!20!,font=\sffamily\Large\bfseries]
		\centering
		\begin{tikzpicture}[scale=1]
		\tikzstyle{edge_style} = [draw=black, line width=2mm, ]
		\tikzstyle{node_style} = [draw=blue,fill=blue!00!,font=\sffamily\Large\bfseries]
		\node (B1) at (.75,-1.3)   {$\bf{E_1}$};
		\node (B1) at (3.75,-1.3)   {$\bf{E_2}$};
		\node (B1) at (6.75,-1.3)   {$\bf{E_3}$};
		\fill[red!100!](0,0) circle (.07);
		\fill[black!100!](1.5,0) circle (.07);
		\fill[black!100!](1.5,1.5) circle (.07);
		\fill[red!100!](0,1.5) circle (.07);
		\fill[black!100!](.75,2) circle (.07);
		\fill[black!100!](.75,-.5) circle (.07);
		%
		%
		\draw[line width=.2 mm] (0,0) -- (0,1.5);
		\draw[line width=.2 mm] (1.5,1.5) -- (1.5,0);
		\draw[line width=.2 mm] (.75,2) -- (0,1.5);
		\draw[line width=.2 mm] (.75,2) -- (1.5,1.5);
		\draw[line width=.2 mm] (.75,-.5) -- (1.5,0);
		\draw[line width=.2 mm] (.75,-.5) -- (0,0);
		\fill[red!100!](3,0) circle (.07);
		\fill[black!100!](4.5,0) circle (.07);
		\fill[black!100!](4.5,1.5) circle (.07);
		\fill[red!100!](3,1.5) circle (.07);
		\fill[black!100!](3.75,2) circle (.07);
		\fill[black!100!](3.75,-.5) circle (.07);
		%
		%
		\draw[line width=.2 mm] (3,0) -- (3,1.5);
		\draw[line width=.2 mm] (4.5,1.5) -- (4.5,0);
		\draw[line width=.2 mm] (3.75,2) -- (3,1.5);
		\draw[line width=.2 mm] (3.75,2) -- (4.5,1.5);
		\draw[line width=.2 mm] (3.75,-.5) -- (4.5,0);
		\draw[line width=.2 mm] (3.75,-.5) -- (3,0);
		\draw[line width=.2 mm] (3.75,-1) -- (4.5,0);
		\draw[line width=.2 mm] (3.75,-1) -- (3,0);
		\fill[black!100!](3.75,-1) circle (.07);
		%
		\draw[line width=.2 mm] (3.4,1.4) -- (3.75,2);
		\draw[line width=.2 mm] (3.4,1.4) -- (3,0);
		\fill[black!100!](3.4,1.4) circle (.07);
		%
		\draw[line width=.2 mm] (3.9,.3) -- (3.75,-.5);
		\draw[line width=.2 mm] (3.9,.3) -- (4.5, 1.5);
		\fill[black!100!](3.9,.3) circle (.07);
		%
		%
		\fill[red!100!](6,0) circle (.07);
		\fill[black!100!](7.5,0) circle (.07);
		\fill[black!100!](7.5,1.5) circle (.07);
		\fill[red!100!](6,1.5) circle (.07);
		\fill[black!100!](6.75,2) circle (.07);
		\fill[black!100!](6.75,-.5) circle (.07);
		%
		%
		\draw[line width=.2 mm] (6,0) -- (6,1.5);
		\draw[line width=.2 mm] (7.5,1.5) -- (7.5,0);
		\draw[line width=.2 mm] (6.75,2) -- (6,1.5);
		\draw[line width=.2 mm] (6.75,2) -- (7.5,1.5);
		\draw[line width=.2 mm] (6.75,-.5) -- (7.5,0);
		\draw[line width=.2 mm] (6.75,-.5) -- (6,0);
		\draw[line width=.2 mm] (6.75,.5) -- (6,0);
		\draw[line width=.2 mm] (6.75,.5) -- (7.5,0);
		\fill[black!100!](6.75,.5) circle (.07);
		\draw[line width=.2 mm] (6.75,.8) -- (6,0);
		\draw[line width=.2 mm] (6.75,.8) -- (7.5,0);
		\fill[black!100!](6.75,.8) circle (.07);
		\draw[line width=.2 mm] (6.75,1.2) -- (6,0);
		\draw[line width=.2 mm] (6.75,1.2) -- (7.5,0);
		\fill[black!100!](6.75,1.2) circle (.07);
		\draw[line width=.2 mm] (6.75,1.5) -- (6,0);
		\draw[line width=.2 mm] (6.75,1.5) -- (7.5,0);
		\fill[black!100!](6.75,1.5) circle (.07);
		\draw[line width=.2 mm] (6.75,2.3) -- (7.5,1.5);
		\draw[line width=.2 mm] (6.75,2.3) -- (6,1.5);
		\fill[black!100!](6.75,2.3) circle (.07);
		\draw[line width=.2 mm] (6.75,2.6) -- (7.5,1.5);
		\draw[line width=.2 mm] (6.75,2.6) -- (6,1.5);
		\fill[black!100!](6.75,2.6) circle (.07);
		\end{tikzpicture}
		\caption{ $E_1$ is the cycle $C_6, E_2$ is the graph obtained from $E_6$ after duplicating three degree-$2$ vertices and $E_3$ is the graph obtained from $C_6$ after duplicating six degree-$2$ vertices}
		\label{fig:graph of diam 3; duplicating cycle c6}	
\end{figure}
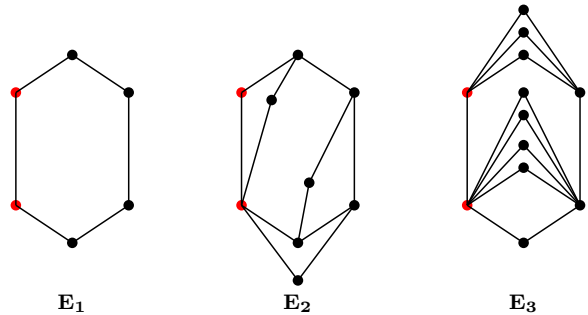
\begin{example}\label{exm: diam 3}
Let $\mathcal{G}$ be the family of graphs that; 
\begin{enumerate}
\item	
contains the cycle $C_6,$ and
\item
is closed under degree-$2$ vertex duplication, (see Figure \ref{fig:graph of diam 3; duplicating cycle c6}). 
\end{enumerate}
Note that the family $\mathcal{G}$ contains infinitely many graphs. Also, for any  $\Gamma\in \mathcal{G},$ the diameter of the graph $\Gamma$ is equal to $3,$ (by Lemma
\ref{lem: diam fixed by multiplication vertices}). Moreover, look at the two red colored vertices (say $u$ and $v$) in each of the graph in Figure \ref{fig:graph of diam 3; duplicating cycle c6}, they are edge connected and it is easy to see that $\text{nbd}(u)\cap\text{nbd}(v)=\emptyset.$ Therefore, by Lemma \ref{lem: nbd umeet nbd v empty after multiplication} and Theorem \ref{thm: v edge u; nbd u meet nbd v=null}, for each $\Gamma\in \mathcal{G},$ the adjacency matrix $A(\Gamma)$ satisfies the conjecture.
\end{example}
The next result is related to the simple connected graph having exactly one dominating vertex. We use this result to prove Theorem \ref{thm: rank 5 graph follows conj}.
\begin{lemma}\label{lemm: one domand other degree m}
Let $\Gamma$ be a non-complete graph with $V(\Gamma)=\{v_1, \cdots, v_n \}.$ Suppose $\Gamma$ has exactly one dominating vertex $v_1$ (say) and $\text{deg}(v_i)=d (1\leq d\leq n-2)$ for all $i\in \{2, \cdots, n\}.$ Then there exists a non-zero $(0,1)$-vector in the row space of $A(\Gamma)$ (over the real numbers) which does not occur as a row of $A(\Gamma).$
\end{lemma}	
\begin{proof}
Here we show that the vector $(\underbrace{1, \cdots, 1}_{n \text{ entries}})\in R(A(\Gamma)).$ First, we arrange the vertices in the order $v_1, v_2, \cdots, v_n.$	
Then the adjacency matrix 
\begin{equation*}\label{eq: adj mat of gen of wheel graph, etc}
A(\Gamma)=\left(
		\begin{array}{c|ccc}
		0 &1 & \cdots&1\\
		\hline
		1 &\\
		\vdots && B\\
		1 &
		\end{array}
		\right),
		\end{equation*}
where $B$ is an $(n-1)$-order $(0, 1)$ matrix having exactly $(d-1)$ one in each column (as $\text{deg}(v_i)=d$ and $(1, i)^{\text{th}}$ entry of $A(\Gamma)$ is $1,$ for each $i\in \{2, \cdots, n\}).$ Therefore,
\begin{equation}
\frac{n-d}{n-1}R_1+\frac{1}{n-1}\left(R_2+\cdots+R_n\right)=(1, \cdots, 1). 
\end{equation}
Hence the lemma.	
\end{proof}
\begin{example} $W_{n+1}$ to denote a wheel graph with $n+1$ vertices $(n\geq3),$ which is formed by connecting a single vertex to all vertices of a cycle of length $n.$ By $T_{n+1},$ we denote the graph with $n+1$ vertices $(n\geq 3  \text{ and  even}),$ which is made by meeting $\frac{n}{2}$ triangles to a single a vertex, (see Figure  \ref{fig:wheel graph}).  By Lemma \ref{lemm: one domand other degree m}, the vector $(\underbrace{1, \cdots, 1}_{n+1 \text{ times}})\in R(A(W_{n+1}))$ and $R(A(T_{n+1})).$
	\end{example}
	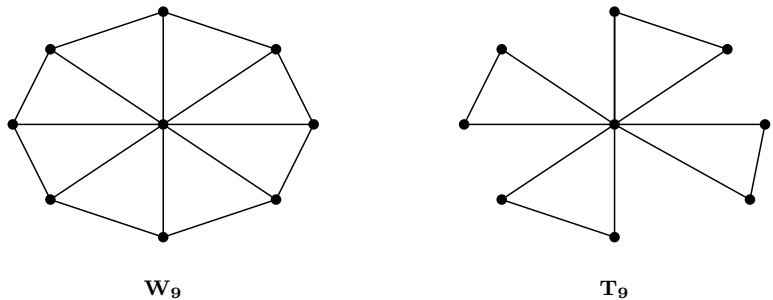
\begin{figure}
		\tiny
		\tikzstyle{ver}=[]
		\tikzstyle{vert}=[circle, draw, fill=black!100, inner sep=0pt, minimum width=4pt]
		\tikzstyle{vertex}=[circle, draw, fill=black!00, inner sep=0pt, minimum width=4pt]
		\tikzstyle{edge} = [draw,thick,-]
		\tikzstyle{node_style} = [circle,draw=blue,fill=blue!20!,font=\sffamily\Large\bfseries]
		\centering
		\begin{tikzpicture}[scale=1]
		\tikzstyle{edge_style} = [draw=black, line width=2mm, ]
		\tikzstyle{node_style} = [draw=blue,fill=blue!00!,font=\sffamily\Large\bfseries]
		\fill[black!100!](-1.5,1) circle (.07);
		\fill[black!100!](1.5,-1) circle (.07);
		\fill[black!100!](-1.5,-1) circle (.07);
		\fill[black!100!](1.5,1) circle (.07);
		\fill[black!100!](0,-1.5) circle (.07);
		\fill[black!100!](0,1.5) circle (.07);
		\fill[black!100!](-2,0) circle (.07);
		\fill[black!100!](2,0) circle (.07);
		\fill[black!100!](0,0) circle (.07);
		\draw[line width=.2 mm] (0,0) -- (1.5,-1);
		\draw[line width=.2 mm] (0,0) -- (-1.5,-1);
		\draw[line width=.2 mm] (0,0) -- (-1.5,1);
		\draw[line width=.2 mm] (0,0) -- (1.5,1);
		\draw[line width=.2 mm] (0,0) -- (0,-1.5);
		\draw[line width=.2 mm] (0,0) -- (0,1.5);
		\draw[line width=.2 mm] (0,0) -- (-2,0);
		\draw[line width=.2 mm] (0,0) -- (2,0);
		\draw[line width=.2 mm] (2,0) -- (1.5,1);
		\draw[line width=.2 mm] (1.5,1) -- (0,1.5);	
		\draw[line width=.2 mm] (-1.5,1) -- (0,1.5);	
		\draw[line width=.2 mm] (-1.5,1) -- (-2,0);	
		\draw[line width=.2 mm] (-2,0) -- (-1.5,-1);	
		\draw[line width=.2 mm] (-1.5,-1) -- (0,-1.5);	
		\draw[line width=.2 mm] (1.5,-1) -- (0,-1.5);
		\draw[line width=.2 mm] (1.5,-1) -- (2,0);
		\node (B1) at (0,-2.2)   {$\bf{W_{9}}$};
		\node (B1) at (6,-2.2)   {$\bf{T_9}$};
		\draw[line width=.2 mm] (6,0) -- (4.5,1);
		\draw[line width=.2 mm] (6,0) -- (7.8,-1);
		\draw[line width=.2 mm] (6,0) -- (6,1.5);
		\draw[line width=.2 mm] (6,0) -- (6,-1.5);
		\draw[line width=.2 mm] (6,0) -- (4.5,-1);
		\draw[line width=.2 mm] (6,0) -- (7.5,1);
		\draw[line width=.2 mm] (6,0) -- (6,1.5);
		\draw[line width=.2 mm] (6,0) -- (4,0);
		\draw[line width=.2 mm] (6,0) -- (8,0);
		\draw[line width=.2 mm] (6,-1.5) -- (4.5,-1);
		\draw[line width=.2 mm] (4,0) -- (4.5,1);
		\draw[line width=.2 mm] (7.5,1) -- (6,1.5);
		\draw[line width=.2 mm] (8,0) -- (7.8,-1);
		\fill[black!100!](6,0) circle (.07);
		\fill[black!100!](8,0) circle (.07);
		\fill[black!100!](4,0) circle (.07);
		\fill[black!100!](6,1.5) circle (.07);
		\fill[black!100!](7.5,1) circle (.07);
		\fill[black!100!](4.5,-1) circle (.07);
		\fill[black!100!](6,-1.5) circle (.07);
		\fill[black!100!](7.8,-1) circle (.07);
		\fill[black!100!](4.5,1) circle (.07);
		\end{tikzpicture}
		\caption{Graphs with a dominating vertex}
		\label{fig:wheel graph}	
	\end{figure}
With the results (Theorem \ref{thm: v edge u; nbd u meet nbd v=null} and Lemma \ref{lemm: one domand other degree m}) proved above we can now prove Theorems \ref{thm: rank 5 graph follows conj}, \ref{thm: dim 2 graph m=2n-5}.
\begin{proof}[Proof of Theorem \ref{thm: rank 5 graph follows conj}]
We divide the proof of this theorem into three cases.
		
Case 1. Let $r(\Gamma)\leq 3.$ Then by Theorem \ref{thm: cha of graphs rank leq 3} and Lemma \ref{lem: diam of rank leq 3 graphs}, $\Gamma\in \mathcal{M}(K_2, K_3)=\mathcal{M}(K_2)\cup \mathcal{M}(K_3)$ and $\text{diam}(\Gamma)\leq 2.$ 
First let $\Gamma\in\mathcal{M}(K_3).$ Then $\Gamma=K_3\bigodot m,$ for some $m=(m_1, m_2, m_3).$ Putting $n=3$ in Equation \eqref{dis: complete graph satis Conj}, we get $\frac{1}{2}\sum\limits_{i=1}^3R_i=(1,1,1)\in R(A(K_3)).$ Thus, by Lemma \ref{lem: v is in row space of graph impl v,some extra is in row space of graph with multiplication}, the vector \[V_1=(\underbrace{1, \cdots, 1}_{t_1 \text{ times} })\in R(A(\Gamma)), \text{ where } t_1=\sum\limits_{i=1}^3m_i.\] If $\Gamma\in \mathcal{M}(K_2),$ similarly, we can show that $A(\Gamma)$ satisfies the conjecture.  
		
Case 2. Let $r(\Gamma)=4.$ Lemma \ref{lem: char of graph of rank 4 with diam 2 and 3} shows that \[\Gamma\in \mathcal{M}(B_1, B_2, B_3, B_4, B_5, P_4, K_4)=\mathcal{M}(B_1, B_2, B_3, B_4, B_5, P_4)\cup\mathcal{M}(K_4).\] Observe that the graphs $B_1, B_2, B_3, B_4, B_5$ and $P_4$ in Figure \ref{fig:rank 4 graphs}. Each of these graphs has at least one pair of vertics (colored by red) $u, v$ (say) such that $u\sim v$ and $\text{nbd}(u)\cap\text{nbd}(v)=\emptyset.$ As a result, $\Gamma$ has two vertices $u', v'$ such that $u'\sim v'$ and $\text{nbd}(u')\cap\text{nbd}(v')=\emptyset$ (by Lemma \ref{lem: nbd umeet nbd v empty after multiplication}). Consequently, Theorem \ref{thm: v edge u; nbd u meet nbd v=null} proves that $A(\Gamma)$ satisfies Conjecture \ref{conj: Cameron on row space}. 
		
Let $\Gamma\in \mathcal{M}(K_4),$ then similarly as the case $\Gamma\in\mathcal{M}(K_3),$ it an be shown that 
\begin{equation*}
V_2=(\underbrace{1, \cdots, 1}_{t_2 \text{ times} })\in R(A(\Gamma)), 
\end{equation*}
$\text{ where } \Gamma=K_4\bigodot m, m=(m_1, \cdots, m_4), \text{ and } t_2=\sum\limits_{i=1}^4m_i.$

Case 3. Let $r(\Gamma)=5.$ Then, by Lemma \ref{lem: ch of graphs of rank five graphs and diam leq 3} $\Gamma\in \mathcal{M}(D_2, \cdots, D_{24})=\mathcal{X}\cup\mathcal{Y},$ where  
\begin{align*}
\mathcal{X}=&\mathcal{M}(D_2, \cdots, D_{5}, D_7, D_8, D_{9}, D_{11}, D_{12}, D_{13}, D_{16}, D_{19}, D_{20}, D_{23}) \text{ and }\\
\mathcal{Y}=&\mathcal{M}(D_6, D_{10}, D_{14}, D_{15}, D_{17}, D_{18}, D_{21}, D_{22}, D_{24}).
\end{align*}
Suppose $\Gamma\in \mathcal{M}(D_2, \cdots, D_{5}, D_7, D_8, D_{9}, D_{11}, D_{12}, D_{13}, D_{16}, D_{19}, D_{20}, D_{23}).$ Look at the graphs $D_2, \cdots, D_{5}, D_7, D_8, D_{9}, D_{11}, D_{12}, D_{13}, D_{16}, D_{19}, D_{20}, D_{23}$ depicted in Figure \ref{fig:rank 5 graphs}. Each of these graphs contains at least one pair of vertices (colored by red) $u, v$ (say) such that $u\sim v$ and $\text{nbd}(u)\cap\text{nbd}(v)=\emptyset.$ So, applying Lemma \ref{lem: nbd umeet nbd v empty after multiplication} and Theorem \ref{thm: v edge u; nbd u meet nbd v=null}, we can  say that $A(\Gamma)$ satisfies the conjecture. 
Now, let
\begin{align*}
\Gamma\in \mathcal{M}(D_6,& D_{10}, D_{14}, D_{15}, D_{17}, D_{18}, D_{21}, D_{22}, D_{24})\\&= \mathcal{M}(D_{10}, D_{21}, D_{22}, D_{24})\cup\mathcal{M}(D_6,  D_{14}, D_{15}, D_{17}, D_{18}).
\end{align*} 
Note that, the adjacency matrices $A(D_{10}), A(D_{21}), A(D_{22}), A(D_{24})$ are non singular (as they achieve the full rank). As a result, the vector  $(1, 1, 1, 1, 1)\in R(A(D_{10})), R(A(D_{21})), R(A(D_{22})),$ and $R(A(D_{24})).$ So, by Lemma \ref{lem: v is in row space of graph impl v,some extra is in row space of graph with multiplication}, the vector \[V_3=(\underbrace{1, \cdots, 1}_{t_3 \text{ times} })\in R(A(\Gamma)),\]  where $\Gamma=D_i\bigodot m,$ for all $i\in\{10, 21, 22, 24\}, m=(m_1, \cdots, m_5)$ and $t_3=\sum\limits_{i=1}^5m_i.$  
That is, $A(\Gamma)$ satisfies the conjecture whenever $\Gamma\in \mathcal{M}(D_{10}, D_{21}, D_{22}, D_{24}).$   
		
Let $\Gamma\in\mathcal{M}(D_6,  D_{14}, D_{15}, D_{17}, D_{18}).$ First, we show that the adjacency matrices \[A(D_6), A(D_{14}), A(D_{15}), A(D_{17}), A(D_{18})\] satisfy the conjecture. Note that, the  graph $D_{18}$ has a unique vertex of degree $6$ (dominating vertex) and the degree of all other vertices are $4.$ So, by Lemma \ref{lemm: one domand other degree m}, the vector $(1, 1, 1, 1, 1, 1, 1)\in R(A(D_{18})).$ Now, look at the remaining adjacency matrices (below).
\begin{equation*}
A(D_6)=\left(
		\begin{array}{ccccccc}
		0 &1 &0&0&0&1&0\\
		1 &0 &1&0&0&1&1\\
		0 &1 &0&1&0&1&1\\
		0 &0 &1&0&1&1&1\\
		0 &0 &0&1&0&0&1\\
		1 &1 &1&1&0&0&0\\
		0 &1 &1&1&1&0&0
		\end{array}
		\right),
		A(D_{14})=\left(
		\begin{array}{cccccc}
		0 &1 &0&0&1&1\\
		1 &0 &1&0&1&0\\
		0 &1 &0&1&1&1\\
		0 &0 &1&0&0&1\\
		1 &1 &1&0&0&1\\
		1 &0 &1&1&1&0\\
		\end{array}
		\right)
		\end{equation*}
		\begin{equation*}
		A(D_{15})=\left(
		\begin{array}{cccccc}
		0 &1 &0&0&1&1\\
		1 &0 &1&0&1&1\\
		0 &1 &0&1&0&1\\
		0 &0 &1&0&1&1\\
		1 &1 &0&1&0&1\\
		1 &1 &1&1&1&0\\
		\end{array}
		\right),
		A(D_{17})=\left(
		\begin{array}{ccccccc}
		0 &1 &0&0&1&1&1\\
		1 &0 &1&0&1&0&1\\
		0 &1 &0&1&1&1&1\\
		0 &0 &1&0&0&1&1\\
		1 &1 &1&0&0&1&0\\
		1 &0 &1&1&1&0&1\\
		1 &1 &1&1&0&1&0
		\end{array}
		\right).
		\end{equation*}
		It is easy to see that
		\begin{align*}\label{eq: D6,D14,D15,D17 satifies conj}
		\frac{1}{2}\left(-R_2+R_3+2R_4+R_6\right)=&(0, 1, 1, 1, 1, 1, 1)\in R(A(D_6))\\
		\frac{1}{2}\left(R_1-R_2+R_5+2R_6\right)=&(1, 1, 1, 1, 1, 1)\in R(A(D_{14}))\\
		\frac{1}{2}\left(R_4+R_5+R_6\right)=&(1, 1, 1, 1, 1, 1) \in R(A(D_{15}))\\
		\frac{1}{2}\left(-R_1+R_2+2R_3+R_5\right)=&(1, 1, 1, 1, 1, 1, 1)\in R(A(D_{17})).
		\end{align*} 
		So, applying Lemma \ref{lem: v is in row space of graph impl v,some extra is in row space of graph with multiplication}, we can say that $\Gamma$ satisfies Conjecture \ref{conj: Cameron on row space}. This completes the proof.
	\end{proof}
	\begin{proof}[Proof of Theorem \ref{thm: dim 2 graph m=2n-5}]
Let $\Gamma$ be a diameter-$2$ graph of order $n$ size $s$ with no dominating vertex. Suppose $s=2n-5.$ Then by Theorem \ref{thm H-M: dim2 m  geq 2n-5}, $\Gamma\in \mathcal{H}.$ Now, each of the graphs in Figure \ref{fig:graph of diam 2;cycle 0f length $5$ and its duplication}, has two red colored vertices say $u$ and $v$ such that $u\sim v$ and $\text{nbd}(u)\cap\text{nbd}(v)=\emptyset.$ Therefore by Lemma \ref{lem: nbd umeet nbd v empty after multiplication} and Theorem \ref{thm: v edge u; nbd u meet nbd v=null}, there exists a non-zero $(0,1)$-vector in the row space of $A(\Gamma)$ (over the real numbers) which does not occur as a row of $A(\Gamma).$ This completes the proof.
\end{proof}	
	
\subsection*{Acknowledgment} I thank Prof. Arvind Ayyer for many helpful discussions.

Data Availability Statement: The availability of data and materials is not applicable.

Conflict of Interest: The authors have not disclosed any competing interests.	
	\bibliographystyle{amsplain}
	\bibliography{gen-inv-lcp.bib}


\end{document}